\documentclass{amsart}

\usepackage[english]{babel}
\usepackage[T1]{fontenc}
\usepackage{ae}
\usepackage[mathscr]{eucal}
\usepackage{amsthm}
\usepackage{amsfonts}
\usepackage{amsmath}
\usepackage{amssymb}
\usepackage[hmarginratio = 1:1]{geometry}
\usepackage{layout}
\usepackage{enumitem}
\usepackage[hidelinks]{hyperref}
\usepackage[utf8]{inputenc}
\usepackage{indentfirst}
\usepackage[dvipsnames]{xcolor}
\usepackage{mathrsfs}

\newtheorem*{theorem a}{Theorem A}
\newtheorem*{theorem*}{Theorem}
\newtheorem{theorem}{Theorem}[section]
\newtheorem{lemma}[theorem]{Lemma}
\newtheorem{corollary}[theorem]{Corollary}
\newtheorem{proposition}[theorem]{Proposition}

\theoremstyle{definition}

\theoremstyle{remark}

\numberwithin{equation}{section}

\begin{document}

\title[Rieffel's deformed algebra as Heisenberg smooth operators]{On Rieffel's conjecture characterizing a deformed algebra as Heisenberg smooth operators}

\author{Rodrigo A. H. M. Cabral}
\address{Mathematics Department, Institute of Mathematics and Statistics, University of S\~ao Paulo (IME-USP), BR-05508-090, S\~ao Paulo, SP, Brazil.}
\email{rahmc@ime.usp.br; rodrigoahmc@gmail.com}

\author{Severino T. Melo}
\address{Mathematics Department, Institute of Mathematics and Statistics, University of S\~ao Paulo (IME-USP), BR-05508-090, S\~ao Paulo, SP, Brazil.}
\email{toscano@ime.usp.br}

\dedicatory{This paper is dedicated to the memory of Marcela I.~Merklen (1969--2018)}

\begin{abstract}
Let $\mathscr{A}$ be a unital C$^*$-algebra and $E_n$ be the Hilbert $\mathscr{A}$-module defined as the completion of the $\mathscr{A}$-valued Schwartz function space $\mathcal{S}^\mathscr{A}(\mathbb{R}^n)$ with respect to the norm $\|f\|_2 := \left\| \int_{\mathbb{R}^n} f(x)^*f(x) \, dx \right\|_\mathscr{A}^{1 / 2}$. Also, let $\text{Ad }\mathcal{U}$ be the canonical action of the $(2n + 1)$-dimensional Heisenberg group by conjugation on the algebra of adjointable operators on $E_n$ and let $J$ be a skew-symmetric linear transformation on $\mathbb{R}^n$. We characterize the smooth vectors under $\text{Ad }\mathcal{U}$ which commute with a certain algebra of right multiplication operators $R_h$, with $h \in \mathcal{S}^\mathscr{A}(\mathbb{R}^n)$, where the product is ``twisted'' with respect to $J$ according to a deformation quantization procedure introduced by M.A.~Rieffel. More precisely, we establish that they coincide with an algebra of left multiplication operators and show that this solves, in particular, a conjecture posed by Rieffel. 
\end{abstract}



\maketitle

\section{Introduction}

In his AMS Memoir \cite{rieffel}, M.A.~Rieffel defined an algebra of pseudodifferential operators with C$^*$-algebra valued symbols, and used it to construct a strict deformation quantization of a C$^*$-algebra with an action of $\mathbb{R}^n$ with respect to a given skew-symmetric linear transformation. We start by describing Rieffel's algebra.

Given a C$^*$-algebra $\mathscr{A}$ with C$^*$-norm $\|\, \cdot \,\|_\mathscr{A}$ and unit $1_\mathscr{A}$, we denote by $\mathcal{S}^\mathscr{A}(\mathbb{R}^n)$ the space of all $\mathscr{A}$-valued smooth functions on $\mathbb{R}^n$ which, together with all their partial derivatives, decay to 0 at infinity more rapidly than the inverse of any polynomial on $\mathbb{R}^n$. The function space $\mathcal{S}^\mathscr{A}(\mathbb{R}^n)$ has a canonical Fr\'echet space structure defined by the norms
\begin{equation}
p_{\alpha, \beta}(f) := \sup_{x \in \mathbb{R}^n} \|x^\alpha \, \partial_x^\beta f(x)\|_\mathscr{A}, \qquad f \in \mathcal{S}^\mathscr{A}(\mathbb{R}^n), \, \alpha, \beta \in \mathbb{N}^n.
\end{equation}
We also define $\mathcal{B}^\mathscr{A}(\mathbb{R}^n)$ as the space of $\mathscr{A}$-valued bounded smooth functions on $\mathbb{R}^n$ whose partial derivatives of all orders are also bounded. When $\mathscr{A} = \mathbb{C}$, we will write simply $\mathcal{S}(\mathbb{R}^n)$ and $\mathcal{B}(\mathbb{R}^n)$, respectively. Another space which plays a central role in \cite{rieffel} and in this paper is the Hilbert $\mathscr{A}$-module $E_n$, defined as the Banach space completion of
$\mathcal{S}^\mathscr{A}(\mathbb{R}^n)$ with respect to the norm
\begin{equation} \label{eq:E_n}
\|f\|_2 := \left\| \int_{\mathbb{R}^n} f(x)^*f(x) \, dx \right\|_\mathscr{A}^{1 / 2}, \qquad f \in \mathcal{S}^\mathscr{A}(\mathbb{R}^n).
\end{equation}
This norm is induced by an $\mathscr{A}$-valued inner product \cite[p.~2]{lance} $\langle \, \cdot \,, \cdot \, \rangle_{E_n}$ on $E_n$, defined as the unique continuous extension to $E_n \times E_n$ of the map
\begin{equation}
(f, g) \longmapsto \int_{\mathbb{R}^n} f(x)^*g(x) \, dx
\end{equation}
on $\mathcal{S}^\mathscr{A}(\mathbb{R}^n) \times \mathcal{S}^\mathscr{A}(\mathbb{R}^n)$. The space $L^2(\mathbb{R}^n, \mathscr{A})$ of (equivalence classes of) square-integrable functions will also play a central role in Section \ref{sect:approximation}. Finally, we denote the C$^*$-algebra of (continuous) \textit{adjointable} operators on the Hilbert $\mathscr{A}$-module $E_n$ by $\mathcal{L}_\mathscr{A}(E_n)$ \cite[p.~9]{lance}, and its usual operator C$^*$-norm by $\|\, \cdot \,\|$.

We will now define Rieffel's pseudodifferential operators (for more details, see the Appendix \ref{sect:rieffel}). Given any skew-symmetric linear transformation $J$ on $\mathbb{R}^n$ and $f \in \mathcal{B}^\mathscr{A}(\mathbb{R}^n)$, the linear operator given by the iterated integral
\begin{equation} \label{rieffelop}
L_f(g)(x) := \int_{\mathbb{R}^n} \left( \int_{\mathbb{R}^n} f(x + J\xi) \, g(x + y) \, e^{2 \pi i \langle \xi, y \rangle} \, dy \right) d\xi, \qquad g \in \mathcal{S}^\mathscr{A}(\mathbb{R}^n), \, x \in \mathbb{R}^n,
\end{equation}
maps $\mathcal{S}^\mathscr{A}(\mathbb{R}^n)$ into $\mathcal{S}^\mathscr{A}(\mathbb{R}^n)$ \cite[Proposition 3.3, p.~25]{rieffel}, satisfies $\langle L_f(g), h \rangle_{E_n} = \langle g, L_{f^*}(h) \rangle_{E_n}$, for all $g, h \in \mathcal{S}^\mathscr{A}(\mathbb{R}^n)$, \cite[Proposition 4.2, p.~30]{rieffel} and extends to a bounded operator on the Hilbert $\mathscr{A}$-module $E_n$ \cite[Theorem 4.6 \& Corollary 4.7, p.~34]{rieffel}. By the continuity of the $\mathscr{A}$-valued inner product we see that this extension, also denoted by $L_f$, is an adjointable operator on $E_n$ satisfying $(L_f)^* = L_{f^*}$.

Given $f, g \in \mathcal{B}^\mathscr{A}(\mathbb{R}^n)$, we have the identity $L_f L_g = L_{f \times_J g}$, where $f \times_J g \in \mathcal{B}^\mathscr{A}(\mathbb{R}^n)$ is defined via \textit{Rieffel's deformed product} \cite[p.~23]{rieffel}
\begin{equation} \label{rieffelprod}
(f \times_J g)(x) := \int_{\mathbb{R}^n} \int_{\mathbb{R}^n} f(x + J\xi) \, g(x + y) \, e^{2 \pi i \langle \xi, y \rangle} \, dy \, d\xi, \qquad x \in \mathbb{R}^n.
\end{equation}
We shall refer to $\left\{ L_f: f \in \mathcal{B}^\mathscr{A}(\mathbb{R}^n) \right\}$ as \textit{Rieffel's deformed algebra}. The integral signs, above, do not denote true integrals (not even iterated integrals), but rather \textit{oscillatory integrals} (see \cite[Chapter 1]{rieffel} \cite[pp.~66--69]{cordes}). If $f \in \mathcal{B}^\mathscr{A}(\mathbb{R}^n)$ and $g \in \mathcal{S}^\mathscr{A}(\mathbb{R}^n)$, then Equation \eqref{rieffelprod} coincides with \eqref{rieffelop}, so that $L_f(g) = f \times_J g$. Moreover, the operator of right multiplication by $g \in \mathcal{B}^\mathscr{A}(\mathbb{R}^n)$, defined by $R_g(f) := f \times_J g$, $f \in \mathcal{S}^\mathscr{A}(\mathbb{R}^n)$, is also (the restriction of) an adjointable operator on $E_n$.

Let $C^\infty(\mathbb{R}^{2n},\mathscr{A})$ be the space of smooth $\mathscr{A}$-valued functions on $\mathbb{R}^{2n}$. We call a linear operator $A \colon \mathcal{S}^\mathscr{A}(\mathbb{R}^n) \longrightarrow \mathcal{S}^\mathscr{A}(\mathbb{R}^n)$ a \textit{pseudodifferential operator} with symbol $a \in C^\infty(\mathbb{R}^{2n},\mathscr{A})$, and write $A = \text{Op}(a)$ if, for every $g \in \mathcal{S}^\mathscr{A}(\mathbb{R}^n)$, we have
\begin{equation} \label{eq:pseudos}
A(g)(x) = \frac{1}{(2 \pi)^{n/2}} \int_{\mathbb{R}^n} a(x, \xi) \, \mathcal{F}(g)(\xi) \, e^{i \langle \xi, x \rangle} \, d\xi, \qquad x \in \mathbb{R}^n,
\end{equation}
where $\mathcal{F}$ denotes the Fourier transform (see Equation \eqref{eq:fourier}). Therefore, the operator $L_f$ is a pseudodifferential operator with symbol $(x, \xi) \longmapsto f(x - J\xi / (2 \pi))$ (see Equation \eqref{L1}).

In the scalar case $\mathscr{A} = \mathbb{C}$, when $E_n$ is the usual Hilbert space $L^2(\mathbb{R}^n)$, H.O. Cordes proved \cite{C} \cite[Chapter 8]{cordes} that a bounded operator $A$ on $L^2(\mathbb{R}^n)$ is a smooth vector for the canonical action of the $(2n + 1)$-dimensional Heisenberg group by conjugation if, and only if, $A = \text{Op}(a)$ for some $a \in \mathcal{B}(\mathbb{R}^{2n})$. Motivated by Cordes' ``lovely characterization'', Rieffel conjectured, in the last paragraph of \cite[Chapter 4]{rieffel}, that an adjointable operator $A \in \mathcal{L}_\mathscr{A}(E_n)$ is smooth under the action $\text{Ad }\mathcal{U}$ of the Heisenberg group (see Equation \eqref{eq:heisenbergsmooth}) \textit{and commutes with every $R_g$}, $g \in \mathcal{B}^\mathscr{A}(\mathbb{R}^n)$ if, and only if, $A = L_f$ for some $f \in \mathcal{B}^\mathscr{A}(\mathbb{R}^n)$. That each $L_f$ commutes with every $R_g$ follows immediately from the associative property of Rieffel's product $\times_J$, and the fact that each $L_f$, $f \in \mathcal{B}^\mathscr{A}(\mathbb{R}^n)$, is smooth under the Heisenberg action can be proved using a generalized version of the Calder\'on-Vaillancourt inequality (see \cite[Subsection ``The algebra $\mathcal{B}_J^\mathscr{A}(\mathbb{R}^n)$'']{cabralforgermelo}). Therefore, in the present paper we are mainly interested in investigating the converse part of Rieffel's conjecture \cite[Chapter 4, p.~39]{rieffel}: \\

\begin{enumerate}

\item[] ``... and then let the Heisenberg group act on the whole algebra of bounded operators, $B$, by conjugation. Then the operators which one obtains from $F$'s in $\mathcal{B}$ [via the Kohn-Nirenberg representation, in the scalar case] are exactly the smooth vectors for this action of the Heisenberg group on $B$ by conjugation. Presumably a similar theorem holds in our present context, though the fact that we are using the left regular representation rather than an irreducible representation means that one must replace $B$ by the algebra of operators which commute with the right regular representation...'' \\

\end{enumerate}

The results obtained in \cite{merklen} \cite{melomerklen2} and \cite{melomerklen3} (see also \cite{melomerklen1}) imply that Rieffel's conjecture is true if $\mathscr{A}$ is a matrix algebra $M_k(\mathcal{C})$, for a commutative separable unital C$^*$-algebra $\mathcal{C}$. The main goal of this paper is to prove Rieffel's conjecture for an arbitrary unital C$^*$-algebra $\mathscr{A}$.

The partial results obtained in \cite{melomerklen1} \cite{merklen} \cite{melomerklen2} \cite{melomerklen3} all depend on an interplay between Rieffel's conjecture and generalizations of Cordes' characterization. Here we take a different approach.

We begin by showing in Section \ref{sect:invariant} that $\mathcal{S}^\mathscr{A}(\mathbb{R}^n)$, which by definition is a dense subspace of $E_n$, is left invariant by the operators in the algebra $C^\infty(\text{Ad }\mathcal{U})$ of smooth vectors for the canonical action of the $(2n + 1)$-dimensional Heisenberg group by conjugation (the elements of $C^\infty(\text{Ad }\mathcal{U})$ will be referred to as the \textit{Heisenberg smooth operators} -- see Equation \eqref{eq:heisenbergsmooth} and Proposition \ref{prop:mainthm1}). General results from Lie group representation theory will be applied in order to obtain this conclusion, which will be an important step for the arguments in Section \ref{sect:conjecture}. In the process, we prove in Proposition \ref{prop:s_equals_cinfty} that $\mathcal{S}^\mathscr{A}(\mathbb{R}^n)$ coincides with the subspace of smooth vectors for the canonical unitary representation of the $(2n + 1)$-dimensional Heisenberg group (see Equation \eqref{eq:unitaryrep}), generalizing a result which is well-known in the scalar case $\mathscr{A} = \mathbb{C}$.

In Section \ref{sect:approximation}, we need the hypothesis of $\mathscr{A}$ being unital in order to use the definition and the properties of a certain ``symbol map'' $S \colon C^\infty(\text{Ad }\mathcal{U}) \longrightarrow \mathcal{B}^\mathscr{A}(\mathbb{R}^{2n})$, introduced in \cite[Equation (12)]{melomerklen2}, which adapts a construction given by Cordes. More precisely, for every $A \in C^\infty(\text{Ad }\mathcal{U})$, we show in Proposition \ref{prop:aprox} that $S(A)$ can be pointwise approximated by a sequence $(S(A \circ L_{\tilde{e}_m}))_{m \in \mathbb{N}}$ of functions in $\mathcal{B}^\mathscr{A}(\mathbb{R}^{2n})$, where $(\tilde{e}_m)_{m \in \mathbb{N}}$ is a specific sequence of functions in $\mathcal{S}^\mathscr{A}(\mathbb{R}^n)$. A careful interplay between the spaces $E_n$, $L^2(\mathbb{R}^n, \mathscr{A})$ and $L^2(\mathbb{R}^n) \cdot 1_\mathscr{A}$ (and their corresponding topologies) will be required in the proof. Moreover, once certain key estimates are established, the strategy will be to ``tensor product'' these estimates and go from dimension $n$ to $2n$, so that uniform bounds on the norms of certain tensor product operators may be obtained.

Finally, Section \ref{sect:conjecture} will be devoted to the proof of Theorem A, whose statement will be given next. Denote by $R_n$ the algebra of the right multiplication operators $R_g$ acting on $E_n$, with $g \in \mathcal{S}^\mathscr{A}(\mathbb{R}^n)$, and let $R_n'$ be its commutant inside the C$^*$-algebra $\mathcal{L}_\mathscr{A}(E_n)$, i.e.,
\begin{equation*}
R_n' := \left\{ A \in \mathcal{L}_\mathscr{A}(E_n): A \circ R_g = R_g \circ A, \, g \in \mathcal{S}^\mathscr{A}(\mathbb{R}^n) \right\}.
\end{equation*}
Equation \eqref{eq:left_inverse} shows that the map $\mathcal{R} \circ S|_{C^\infty(\text{Ad }\mathcal{U}) \cap R_n'}$ is a left inverse for $L \colon \mathcal{B}^\mathscr{A}(\mathbb{R}^n) \longrightarrow C^\infty(\text{Ad }\mathcal{U})$, $f \longmapsto L_f$. We will show in Theorem A that $\mathcal{R} \circ S|_{C^\infty(\text{Ad }\mathcal{U}) \cap R_n'}$ is injective. More precisely, we will establish the equality
\begin{equation*}
L \circ \mathcal{R} \circ S|_{C^\infty(\text{Ad }\mathcal{U}) \cap R_n'} = Id_{C^\infty(\text{Ad }\mathcal{U}) \cap R_n'},
\end{equation*}
where $Id_{C^\infty(\text{Ad }\mathcal{U}) \cap R_n'}$ denotes the identity operator on $C^\infty(\text{Ad }\mathcal{U}) \cap R_n'$:

\begin{theorem a} \label{thm:rieffelunital}
Let $\mathscr{A}$ be a unital C$^*$-algebra and let $R_n'$ be the commutant defined above. If $A \in C^\infty(\text{Ad }\mathcal{U}) \cap R_n'$, then $A = L_{(\mathcal{R} \circ S)(A)}$.
\end{theorem a}

The reason we are interested in Theorem A is because it allows us to obtain, as an immediate corollary, a solution to Rieffel's conjecture:

\begin{theorem*}[Rieffel's conjecture for a unital C$^*$-algebra] \label{thm:rieffelunital_conj}
Let $\mathscr{A}$ be a unital C$^*$-algebra. If a Heisenberg smooth operator $A \in \mathcal{L}_\mathscr{A}(E_n)$ commutes with every operator of the form $R_h$, with $h \in \mathcal{B}^\mathscr{A}(\mathbb{R}^n)$, then there exists $f \in \mathcal{B}^\mathscr{A}(\mathbb{R}^n)$ such that $A = L_f$.
\end{theorem*}

\begin{proof}
Let $A \in C^\infty(\text{Ad }\mathcal{U})$. If $A$ commutes with every operator $R_h$, with $h \in \mathcal{B}^\mathscr{A}(\mathbb{R}^n)$, then in particular it commutes with every $R_h$, with $h \in \mathcal{S}^\mathscr{A}(\mathbb{R}^n)$. But, as a consequence of Theorem A, we know that if $A$ commutes with every operator of the form $R_h$, with $h \in \mathcal{S}^\mathscr{A}(\mathbb{R}^n)$, then $A = L_{(\mathcal{R} \circ S)(A)}$. Therefore, since $(\mathcal{R} \circ S)(A)$ belongs to $\mathcal{B}^\mathscr{A}(\mathbb{R}^n)$, the result follows immediately.
\end{proof}

The proof of Theorem A consists essentially of the computation of two limits. For the first one, we make use of Proposition \ref{prop:aprox}, as well as of certain estimates which are obtained in its proof. For the second one, the hypothesis that $A \in C^\infty(\text{Ad }\mathcal{U})$ commutes with every $R_g$, $g \in \mathcal{S}^\mathscr{A}(\mathbb{R}^n)$, is finally considered: we adapt a useful calculation performed in \cite[Corol\'ario 2.7]{merklentese} and combine it with the results of Proposition \ref{prop:mainthm1} and Lemma \ref{lem:approxidentity}. Finally, a careful combination of the conclusions summarized in Equations \eqref{eq:rieffel1} and \eqref{eq:rieffel2} will establish Theorem A.

\section{A (reasonable) dense invariant subspace for the Heisenberg smooth operators} \label{sect:invariant}

We begin this section by recalling a few basic concepts and fixing the notation which will be employed throughout the paper. For more information on Lie group representations, see \cite{poulsen}.

If $G$ is a (finite-dimensional) Lie group with unit $e$ and $(\mathcal{X}, \|\, \cdot \,\|)$ is a Banach space, we say that a family $\left\{ V_g \right\}_{g \in G}$ of continuous (everywhere defined) linear operators on $\mathcal{X}$ is a \textit{strongly continuous representation of $G$ on $\mathcal{X}$} if
\begin{equation*}
V_e = I, \, V_{gh} = V_g V_h, \quad \text{and} \quad \lim_{h' \rightarrow h} V_{h'}x = V_hx, \qquad x \in \mathcal{X}, \, g, h \in G,
\end{equation*}
with $I$ denoting the identity operator on $\mathcal{X}$. If $G = \mathbb{R}$, then we call $\left\{ V_t \right\}_{t \in \mathbb{R}}$ a \textit{strongly continuous one-parameter group}. In this case, the subspace of vectors $x \in \mathcal{X}$ such that $\lim_{t \rightarrow 0} t^{-1}(V_t x - x)$ exists in $\mathcal{X}$ defines the domain $\text{Dom }T$ of a linear operator $T$ on $\mathcal{X}$, $T(x) := \lim_{t \rightarrow 0} t^{-1}(V_t x - x)$, which is called the \textit{infinitesimal generator} of $\left\{ V_t \right\}_{t \in \mathbb{R}}$. Moreover, if $G$ is any Lie group and $\left\{ V_g \right\}_{g \in G}$ is a strongly continuous representation of $G$ on $\mathcal{X}$, then each $X$ in its Lie algebra $\mathfrak{g}$ gives rise to a one-parameter group $t \longmapsto V_{\exp tX}$ on $\mathcal{X}$ ($\exp$ denotes the exponential map of the Lie group $G$), and its infinitesimal generator is denoted by $dV(X)$. Finally, a vector $x \in \mathcal{X}$ is called a \textit{smooth vector for $V$} if the map $G \ni g \longmapsto V_g x$ is of class $C^\infty$ (in other words, if the map $G \ni g \longmapsto V_g x$ has continuous partial derivatives of all orders on an arbitrary chart). The subspace of smooth vectors for $V$ will be denoted by $C^\infty(V)$. It turns out that $C^\infty(V)$ is a dense subspace of $\mathcal{X}$ which is invariant by the operators $\left\{ dV(X) \right\}_{X \in \mathfrak{g}}$, and the map $\partial V \colon X \longmapsto \partial V(X) := dV(X)|_{C^\infty(V)}$ is a Lie algebra representation on $C^\infty(V)$ which extends to a representation of (the complexification of) the universal enveloping algebra of $\mathfrak{g}$. Given a basis $\mathcal{B} := (X_k)_{1 \leq k \leq d}$ for $\mathfrak{g}$ we may equip $C^\infty(V)$ with a Fr\'echet space topology defined by the family
\begin{equation} \label{smoothfrechet}
\left\{ \rho_n: n \in \mathbb{N} \right\}
\end{equation}
of norms, with $\rho_0(x) := \|x\|$, \, $dV(X_0) := I$ and
\begin{equation*}
\rho_n(x) := \max \left\{ \|dV(X_{i_1}) \cdots dV(X_{i_n})x\|: 0 \leq i_j \leq d \right\}, \quad n \geq 1;
\end{equation*}
note, however, that this topology does not depend upon the fixed basis $\mathcal{B}$.

We recall that the Heisenberg group $H_{2n + 1}(\mathbb{R}) = \left\{ (\texttt{a}, \texttt{b}, c): \texttt{a}, \texttt{b} \in \mathbb{R}^n, \, c \in \mathbb{R} \right\}$ (with multiplication given by $(\texttt{a}, \texttt{b}, c) \cdot (\texttt{a}', \texttt{b}', c') := (\texttt{a} + \texttt{a}', \texttt{b} + \texttt{b}', c + c' - \langle \texttt{b}', \texttt{a} \rangle)$) embeds as a subgroup of the group of invertible matrices $GL_{n + 2}(\mathbb{R})$ via the map
$\iota \colon (\texttt{a},\texttt{b},c)\ \longmapsto\ \begin{bmatrix}
1 & \texttt{a}^T & c\\
0 & I_n & -\texttt{b}\\
0 & 0 & 1
\end{bmatrix}$, with $I_n$ denoting the identity matrix of $M_n(\mathbb{R})$. A \textit{unitary operator} on $E_n$ is, by definition, an adjointable operator $u \in \mathcal{L}_\mathscr{A}(E_n)$ satisfying $u^* u = u u^* = I$ \cite[p.~24]{lance} (some might denote ``unitary'' instead of unitary, only to stress that $E_n$ is not a Hilbert space). One may then define a strongly continuous unitary representation $\mathcal{U}$ of $H_{2n + 1}(\mathbb{R})$ on the Hilbert $\mathscr{A}$-module $E_n$: first, we define it on $\mathcal{S}^\mathscr{A}(\mathbb{R}^n)$ by
\begin{equation} \label{eq:unitaryrep}
\mathcal{U}_{\texttt{a}, \texttt{b}, c}(f)(x) := e^{ic} e^{i \langle \texttt{b}, x \rangle} f(x - \texttt{a}), \qquad f \in \mathcal{S}^\mathscr{A}(\mathbb{R}^n), \, x \in \mathbb{R}^n,
\end{equation}
and then extend it by continuity to all of $E_n$. Conjugating with the representation $\mathcal{U}$ gives rise to a (not everywhere strongly continuous) representation of the Heisenberg group $H_{2n + 1}(\mathbb{R})$ on the C$^*$-algebra of adjointable operators $\mathcal{L}_\mathscr{A}(E_n)$ defined by
\begin{equation} \label{eq:heisenbergsmooth}
(\text{Ad }\mathcal{U})_{\texttt{a}, \texttt{b}, c}(\, \cdot \,) := \mathcal{U}_{\texttt{a}, \texttt{b}, c} \, (\, \cdot \,) \, (\mathcal{U}_{\texttt{a}, \texttt{b}, c})^{-1}.
\end{equation}
Note that $(\text{Ad }\mathcal{U})_{\texttt{a}, \texttt{b}, c}$ does not depend on the real variable $c$, so we will simply write $(\text{Ad }\mathcal{U})_{\texttt{a}, \texttt{b}}$. We will restrict our attention to the $\text{Ad }\mathcal{U}$-invariant C$^*$-subalgebra $C(\text{Ad }\mathcal{U})$ of elements $A \in \mathcal{L}_\mathscr{A}(E_n)$ for which $\text{Ad }\mathcal{U}$ is strongly continuous or, in other words, the C$^*$-subalgebra $C(\text{Ad }\mathcal{U})$ of elements $A \in \mathcal{L}_\mathscr{A}(E_n)$ satisfying $\lim_{(\texttt{a}, \texttt{b}) \rightarrow (0_n, 0_n)} (\text{Ad }\mathcal{U})_{\texttt{a}, \texttt{b}}(A) = A$ ($0_n$ denotes the zero vector in $\mathbb{R}^n$). The corresponding smooth vectors for this (slightly modified) representation will be denoted by $C^\infty(\text{Ad }\mathcal{U})$, and its elements will be referred to as \textit{Heisenberg smooth operators}.

Let $(f_k)_{1 \leq k \leq n}$ be the canonical basis of $\mathbb{R}^n$ and let $(X_k)_{1 \leq k \leq 2n + 1}$ be the canonical basis for the Lie algebra of $H_{2n + 1}(\mathbb{R})$, in which $X_k$ is defined as
\begin{equation*}
(f_k, 0_n, 0), \quad \text{if }1 \leq k \leq n, \qquad (0_n, f_{k - n}, 0), \quad \text{if }n + 1 \leq k \leq 2n,
\end{equation*}
or
\begin{equation*}
(0_n, 0_n, 1), \quad \text{if }k = 2n + 1.
\end{equation*}
Then for terminological convenience the restrictions of the infinitesimal generators $d(\text{Ad }\mathcal{U})(X_k)$ to $C^\infty(\text{Ad }\mathcal{U})$, when $1 \leq k \leq 2n$, will be denoted simply by $\partial_k$. Since $C^\infty(\text{Ad }\mathcal{U})$ is a core for each infinitesimal generator, we have $d(\text{Ad }\mathcal{U})(X_k) = \overline{\partial_k}$, for all $1 \leq k \leq 2n$, where $\overline{(\, \cdot \,)}$, as usual, denotes the closure of the operator under consideration.

Before proceeding to the main results of this section, we make a few comments about the Fourier transform $\mathcal{F}$, a Fr\'echet space automorphism of $\mathcal{S}^\mathscr{A}(\mathbb{R}^n)$ defined by
\begin{equation} \label{eq:fourier}
\mathcal{F}(g)(\xi) := \frac{1}{(2 \pi)^{n / 2}} \int_{\mathbb{R}^n} e^{-i \langle s, \xi \rangle} \, g(s) \, ds, \qquad g \in \mathcal{S}^\mathscr{A}(\mathbb{R}^n), \, \xi \in \mathbb{R}^n,
\end{equation}
with inverse given by $\mathcal{F}^{-1}(g)(x) = \mathcal{F}(g)(-x)$ (see \cite[Section 2.4, p.~105]{analysis-bochner}). It extends to an isometry on $E_n$. In fact, for every $f, g \in \mathcal{S}^\mathscr{A}(\mathbb{R}^n)$, we have
\begin{align} \label{plancherel}
(2 \pi)^{n/2} \langle \mathcal{F}(f), g \rangle & = \int \left( \int e^{-i \langle x, y \rangle} f(y) \, dy \right)^* g(x) \, dx \\ & = \int f(y)^* \left( \int e^{i \langle x, y \rangle} g(x) \, dx \right) dy = (2 \pi)^{n/2} \langle f, \mathcal{F}^{-1}(g) \rangle, \nonumber
\end{align}
so substituting $g = \mathcal{F}(f)$ in the above equality shows that $\mathcal{F}$ is an isometry onto $\mathcal{S}^\mathscr{A}(\mathbb{R}^n)$ (with respect to the norm $\|\, \cdot \,\|_2$). Therefore, since $\mathcal{S}^\mathscr{A}(\mathbb{R}^n)$ is dense in $E_n$, we see that $\mathcal{F}$ extends to an isometry onto $E_n$. Moreover, by the continuity of the $\mathscr{A}$-valued inner product, we obtain $\langle \mathcal{F}(f), g \rangle = \langle f, \mathcal{F}^{-1}(g) \rangle$, for $f, g \in E_n$. This shows that $\mathcal{F}$ is an adjointable operator on $E_n$, with $\mathcal{F}^* = \mathcal{F}^{-1}$, i.e., a generalized form of Plancherel's theorem holds for the Hilbert $\mathscr{A}$-module $E_n$.

Now, we focus on the main objective of this section, which is to prove that every Heisenberg smooth operator leaves $\mathcal{S}^\mathscr{A}(\mathbb{R}^n)$ invariant. We begin with an auxiliary lemma, which proves that each such operator maps $\mathcal{S}^\mathscr{A}(\mathbb{R}^n)$ into the space of smooth vectors for the unitary representation $\mathcal{U}$.

\begin{lemma} \label{lem:s_to_cinfty}
Every Heisenberg smooth operator $A \in \mathcal{L}_\mathscr{A}(E_n)$ maps $\mathcal{S}^\mathscr{A}(\mathbb{R}^n)$ into $C^\infty(\mathcal{U})$.
\end{lemma}

\begin{proof}
Let $(f_k)_{1 \leq k \leq n}$ be the canonical basis of $\mathbb{R}^n$ and $f \in \mathcal{S}^\mathscr{A}(\mathbb{R}^n)$. Moreover, let $(X_k)_{1 \leq k \leq 2n + 1}$ be the canonical basis for the Lie algebra of $H_{2n + 1}(\mathbb{R})$, as defined just before Equation \eqref{eq:fourier}.

For each fixed $v \in \mathbb{R}^n$, let $T_v$ denote the continuous extension to $E_n$ of the operator on $\mathcal{S}^\mathscr{A}(\mathbb{R}^n)$ defined by $(T_v \, g)(\, \cdot \,) := g(\, \cdot \, - v)$. As a first step, let us check that the limit
\begin{equation} \label{diff}
\lim_{h \rightarrow 0} \frac{(T_{-h f_k} Af) - Af}{h}
\end{equation}
exists in $E_n$, for all $1 \leq k \leq n$; to simplify notation let us write $h$ for $h f_k$. Write the decomposition
\begin{equation*}
\frac{T_{-h}(Af) - Af}{h} = \underbrace{\frac{(T_{-h} A T_h)f - Af}{h}}_{\text{(I)}} + \underbrace{\frac{(T_{-h} A T_h)(T_{-h} f) - (T_{-h} A T_h)f}{h}}_{\text{(II)}},
\end{equation*}
and let us calculate the limits of the expressions (I) and (II) when $h \rightarrow 0$. We begin with (I), which is simpler: by the definition of $\partial_k$ we have
\begin{equation*}
\lim_{h \rightarrow 0} \left\| \frac{(T_{-h} A T_h)f - Af}{h} - (- \partial_k A)f \right\|_2 \leq \lim_{h \rightarrow 0} \left\| \frac{T_{-h} A T_h - A}{h} - (- \partial_k A) \right\| \|f\|_2 = 0,
\end{equation*}
since $A$ is a Heisenberg smooth operator and $T_{-h} A T_h = (\text{Ad }\mathcal{U})_{-h f_k, 0}(A)$. To analyze (II), we first write the decomposition
\begin{align*}
\frac{(T_{-h} A T_h)(T_{-h} f) - (T_{-h} A T_h)f}{h} & = \underbrace{ \left[ \frac{(T_{-h} A T_h)(T_{-h} f) - (T_{-h} A T_h) f}{h} - (T_{-h} A T_h) \frac{\partial}{\partial x_k} f \right] }_{\text{(III)}} \\ & + \underbrace{\left[ (T_{-h} A T_h) \left( \frac{\partial}{\partial x_k} f \right) - A \left( \frac{\partial}{\partial x_k} f \right) \right]}_{\text{(IV)}} +  A \left( \frac{\partial}{\partial x_k} f \right).
\end{align*}
The expression corresponding to (IV) goes to 0, as $h \rightarrow 0$, because $A$ is a Heisenberg smooth operator so, in particular, $\lim_{h \rightarrow 0} \|T_{-h} A T_h - A\| = 0$. To see that the expression corresponding to (III) also goes to 0, as $h \rightarrow 0$, first note that the sum
\begin{equation*}
\frac{T_{-h} - I}{h}f - \frac{\partial}{\partial x_k} f
\end{equation*}
converges to 0 in $(\mathcal{S}^\mathscr{A}(\mathbb{R}^n), \|\, \cdot \,\|_2)$ (because it does in the natural Fr\'echet topology of $\mathcal{S}^\mathscr{A}(\mathbb{R}^n)$), so
\begin{equation*}
\lim_{h \rightarrow 0} \left\| \frac{(T_{-h} A T_h)(T_{-h} f) - (T_{-h} A T_h)f}{h} - (T_{-h} A T_h) \frac{\partial}{\partial x_k} f \right\|_2 \leq \|A\| \lim_{h \rightarrow 0} \left\| \frac{T_{-h} - I}{h} f - \frac{\partial}{\partial x_k} f \right\|_2 = 0.
\end{equation*}
Therefore, we get
\begin{equation*}
\lim_{h \rightarrow 0} \frac{(T_{-h} A T_h)(T_{-h} f) - (T_{-h} A T_h) f}{h} = A \left( \frac{\partial}{\partial x_k} f \right).
\end{equation*}
Combining the analyses of (I) and (II) proves that the limit \eqref{diff} exists in $E_n$ and satisfies
\begin{equation}
\lim_{h \rightarrow 0} \frac{(T_{-h f_k} Af) - Af}{h} = (- \partial_k A) f + A \left( \frac{\partial}{\partial x_k} f \right). \end{equation}
In particular, this shows that $Af$ belongs to $\text{Dom }d \mathcal{U}(X_k)$, for all $1 \leq k \leq n$.

Now, for each fixed $v \in \mathbb{R}^n$, denote by $M_v$ the continuous extension to $E_n$ of the operator on $\mathcal{S}^\mathscr{A}(\mathbb{R}^n)$ defined by $(M_v \, g)(\, \cdot \,) := e^{i \langle v, \, \cdot \, \rangle} \, g(\, \cdot \,)$. Then we prove with an analogous reasoning that the limit $\lim_{h \rightarrow 0} \frac{(M_{h f_k} Af) - Af}{h}$ exists in $E_n$ and that
\begin{equation}
\lim_{h \rightarrow 0} \frac{(M_{h f_k} Af) - Af}{h} = (\partial_{n + k} A) f + A (i x_k \, f),
\end{equation}
where $i x_k$ denotes the multiplication operator on $\mathcal{S}^\mathscr{A}(\mathbb{R}^n)$ given by $(i x_k \, g)(x) := i x_k \, g(x)$. This shows that $Af$ belongs to $\text{Dom }d \mathcal{U}(X_{n + k})$, for all $f \in \mathcal{S}^\mathscr{A}(\mathbb{R}^n)$ and $1 \leq k \leq n$. 

Clearly, $Af$ also belongs to $\text{Dom }d \mathcal{U}(X_{2n + 1})$, so we conclude that $Af \in \cap_{j = 1}^{2n + 1} \text{Dom }d \mathcal{U}(X_j)$. Since the Heisenberg smooth operator $A \in \mathcal{L}_\mathscr{A}(E_n)$ and the function $f \in \mathcal{S}^\mathscr{A}(\mathbb{R}^n)$ were arbitrary, an inductive procedure on the above calculations shows that every smooth vector $A$ sends $\mathcal{S}^\mathscr{A}(\mathbb{R}^n)$ into
\begin{equation*}
\bigcap_{j = 1}^{2n + 1} \bigcap_{m \in \mathbb{N}} \text{Dom }(d \mathcal{U}(X_j))^m = C^\infty(\mathcal{U})
\end{equation*}
(for this last equality see, for example, \cite[p.~90]{poulsen}).
\end{proof}

It is clear that the inclusion $\mathcal{S}^\mathscr{A}(\mathbb{R}^n) \subseteq C^\infty(\mathcal{U})$ holds. We will now prove that $C^\infty(\mathcal{U})$ actually coincides with the function space $\mathcal{S}^\mathscr{A}(\mathbb{R}^n)$.

\begin{proposition} \label{prop:s_equals_cinfty}
The space of smooth vectors $C^\infty(\mathcal{U})$ for the unitary representation $\mathcal{U}$ coincides with $\mathcal{S}^\mathscr{A}(\mathbb{R}^n)$.
\end{proposition}

\begin{proof}
We begin the proof by showing that the families $\left\{ p_{\alpha, \beta} \right\}_{\alpha, \beta \in \mathbb{N}^n}$ and $\left\{ q_{N_1, N_2} \right\}_{N_1, N_2 \in \mathbb{N}}$ of seminorms defined by
\begin{align} \label{eq:seminorms}
p_{\alpha, \beta}(f) & := \sup_{x \in \mathbb{R}^n} \|x^\alpha \, \partial_x^\beta f(x)\|_\mathscr{A}, \qquad f \in \mathcal{S}^\mathscr{A}(\mathbb{R}^n), \, \alpha, \beta \in \mathbb{N}^n, \\
q_{N_1, N_2}(f) & := \left( \sum_{|\alpha| \leq N_1, \, |\beta| \leq N_2} \|x^\alpha \, \partial_x^\beta f\|_2^2 \right)^{1 / 2}, \qquad f \in \mathcal{S}^\mathscr{A}(\mathbb{R}^n), \, N_1, N_2 \in \mathbb{N}, \nonumber
\end{align}
generate equivalent topologies on $\mathcal{S}^\mathscr{A}(\mathbb{R}^n)$. The topology generated by the family $\left\{ p_{\alpha, \beta} \right\}_{\alpha, \beta \in \mathbb{N}^n}$ is finer than the one generated by $\left\{ q_{N_1, N_2} \right\}_{N_1, N_2 \in \mathbb{N}}$, because using the identity
\begin{equation*}
(1 + |x|^2)^m = \sum_{\gamma \in \mathbb{N}^n, \, |\gamma| \leq m} C_{m, \gamma} \, x^{2 \gamma}, \qquad C_{m, \gamma} := \frac{m!}{(m - |\gamma|)! \, |\gamma|!} \cdot \frac{|\gamma|!}{\gamma_1! \ldots \gamma_n!}
\end{equation*}
we see that, for all $m \in \mathbb{N}$, $m > n / 2$, there exists $C_m > 0$ such that
\begin{align} \label{eq:finer}
\|x^\alpha \, \partial_x^\beta f\|_2^2 & \leq \int_{\mathbb{R}^n} \frac{(1 + |x|^2)^m \, \|x^\alpha \, \partial_x^\beta f(x)\|_\mathscr{A}^2}{(1 + |x|^2)^m} \, dx \\ & \leq \sum_{|\gamma| \leq m} \sup_{x \in \mathbb{R}^n} \|x^{\alpha + 2 \gamma} \, \partial_x^\beta f(x) \|_\mathscr{A}^2 \int_{\mathbb{R}^n} \frac{C_{m, \gamma}}{(1 + |x|^2)^m} \, dx \nonumber \\ & \leq C_m \left( \sum_{|\gamma| \leq m} \sup_{x \in \mathbb{R}^n} \|x^{\alpha + 2 \gamma} \, \partial_x^\beta f(x) \|_\mathscr{A} \right)^2 \! \! \!, \nonumber
\end{align}
for every $f \in \mathcal{S}^\mathscr{A}(\mathbb{R}^n)$ and $\alpha, \beta \in \mathbb{N}^n$.

On the other hand, applying Fourier's Inversion Formula combined with the Cauchy-Schwarz inequality for Hilbert C$^*$-modules \cite[Proposition 1.1, p.~3]{lance} and Equation \eqref{plancherel}, we obtain, for each $m \in \mathbb{N}$ satisfying $m > n/2$, a constant $D_m > 0$ such that the following estimate holds:
\begin{align} \label{eq:coarser}
\|x^\alpha \, \partial_x^\beta f(x)\|_\mathscr{A} & = \|(\mathcal{F}^{-1} \circ \mathcal{F})(y^\alpha \, \partial_y^\beta f)(x)\|_\mathscr{A} \\ & = \frac{1}{(2 \pi)^{n / 2}} \left\| \int_{\mathbb{R}^n} \left[ \frac{e^{i \langle \xi, x \rangle}}{(1 + |\xi|^2)^m} \, 1_\mathscr{A} \right] \, (1 + |\xi|^2)^m \, [\mathcal{F}(y^\alpha \, \partial_y^\beta f)](\xi) \, d\xi \right\|_\mathscr{A} \nonumber \\ & \leq \frac{1}{(2 \pi)^{n / 2}} \, \left\| \frac{e^{i \langle \xi, x \rangle}}{(1 + |\xi|^2)^m} \, 1_\mathscr{A} \right\|_2 \, \left\| (1 + |\xi|^2)^m \, [\mathcal{F}(y^\alpha \, \partial_y^\beta f)](\xi) \right\|_2 \nonumber \\ & \stackrel{\eqref{plancherel}}{=} D_m \, \left\| (1 - \Delta_\xi)^m \, [\xi^\alpha \, \partial_\xi^\beta f](\xi) \right\|_2, \qquad x \in \mathbb{R}^n \nonumber
\end{align}
(note that, in order to guarantee that the function $\xi \longmapsto (e^{i \langle \xi, x \rangle} / (1 + |\xi|^2)^m) \, 1_\mathscr{A}$ belongs to $E_n$, we have used the fact that $L^2(\mathbb{R}^n)$ is isometrically embedded in $E_n$; see \cite[Remark D.2, Appendix D]{cabralforgermelo}). Taking the supremum over all $x \in \mathbb{R}^n$ in \eqref{eq:coarser} shows that the topology generated by the family $\left\{ p_{\alpha, \beta} \right\}_{\alpha, \beta \in \mathbb{N}^n}$ is coarser than the one generated by $\left\{ q_{N_1, N_2} \right\}_{N_1, N_2 \in \mathbb{N}}$. This proves our first claim: the topologies generated by the two sets of seminorms in \eqref{eq:seminorms} are equivalent.

Note that $\mathcal{S}^\mathscr{A}(\mathbb{R}^n) \subseteq C^\infty(\mathcal{U})$ is dense in $E_n$ and $\mathcal{U}_{\texttt{a}, \texttt{b}, c}[\mathcal{S}^\mathscr{A}(\mathbb{R}^n)] \subseteq \mathcal{S}^\mathscr{A}(\mathbb{R}^n)$, for all $\texttt{a}, \texttt{b} \in \mathbb{R}^n$ and $c \in \mathbb{R}$. Therefore, it follows from \cite[Theorem 1.3]{poulsen} that $\mathcal{S}^\mathscr{A}(\mathbb{R}^n)$ is dense in $C^\infty(\mathcal{U})$ with respect to the Fr\'echet topology generated by the family \eqref{smoothfrechet} of norms on $C^\infty(\mathcal{U})$ (with $\|\, \cdot \,\|$ and $\left\{ V_g \right\}_{g \in G}$ replaced by $\|\, \cdot \,\|_2$ and $\left\{ \mathcal{U}_g \right\}_{g \in H_{2n + 1}(\mathbb{R})}$, respectively). The topology induced on $\mathcal{S}^\mathscr{A}(\mathbb{R}^n)$ by the above family of norms on $C^\infty(\mathcal{U})$ coincides with the one generated by the family $\left\{ q_{N_1, N_2} \right\}_{N_1, N_2 \in \mathbb{N}}$ of norms. But, in our first claim we have proved, in particular, that equipping $\mathcal{S}^\mathscr{A}(\mathbb{R}^n)$ with the family $\left\{ q_{N_1, N_2} \right\}_{N_1, N_2 \in \mathbb{N}}$ of norms turns it into a Fr\'echet space, because we already know that $\mathcal{S}^\mathscr{A}(\mathbb{R}^n)$, when equipped with the family $\left\{ p_{\alpha, \beta} \right\}_{\alpha, \beta \in \mathbb{N}^n}$ of seminorms, becomes a Fr\'echet space. This shows that $\mathcal{S}^\mathscr{A}(\mathbb{R}^n)$ is a closed and dense subspace of the Fr\'echet space $C^\infty(\mathcal{U})$, which forces the equality $\mathcal{S}^\mathscr{A}(\mathbb{R}^n) = C^\infty(\mathcal{U})$.
\end{proof}

The main theorem of this section is an immediate corollary of Lemma \ref{lem:s_to_cinfty} combined with Proposition \ref{prop:s_equals_cinfty}:

\begin{proposition} \label{prop:mainthm1}
Every Heisenberg smooth operator $A \in \mathcal{L}_\mathscr{A}(E_n)$ maps $\mathcal{S}^\mathscr{A}(\mathbb{R}^n)$ into $\mathcal{S}^\mathscr{A}(\mathbb{R}^n)$.
\end{proposition}

\section{An approximation theorem for Cordes' symbol map} \label{sect:approximation}

Throughout the rest of the manuscript, we will denote $\mathbb{N}^* := \mathbb{N} \backslash \left\{0\right\}$. Also, since we will need to invoke certain results about vector-valued integration from reference \cite{analysis-bochner}, we shall adopt some of its notations and definitions. Let $\mu$ be the Lebesgue measure on $\mathbb{R}^n$. Then we shall say that a function $f \colon \mathbb{R}^n \longrightarrow \mathscr{A}$ is \textit{$\mu$-simple} if $f(x) = \sum_{j = 1}^N 1_{B_j}(x) \, a_j$, for some fixed natural number $N > 0$ and all $x \in \mathbb{R}^n$, with $a_j$ being certain elements of $\mathscr{A}$ and $1_{B_j}$ the indicator functions of Lebesgue-measurable subsets $B_j$ of $\mathbb{R}^n$ satisfying $\mu(B_j) < + \infty$, for all $1 \leq j \leq N$ \cite[Definition 1.1.13, p.~8]{analysis-bochner}. Moreover, we shall say that a function $f \colon \mathbb{R}^n \longrightarrow \mathscr{A}$ is \textit{strongly $\mu$-measurable} if it is the $\mu$-almost everywhere pointwise limit of a sequence of $\mu$-simple functions \cite[Definition 1.1.14, p.~8]{analysis-bochner}. Finally, the space of equivalence classes of strongly $\mu$-measurable square-integrable $\mathscr{A}$-valued functions on $\mathbb{R}^n$ will be denoted by $L^2(\mathbb{R}^n, \mathscr{A})$ \cite[Definition 1.2.15, p.~21]{analysis-bochner}. It will be customary to write simply $dx$, instead of $d\mu(x)$, when integrating a strongly $\mu$-measurable function $f \colon \mathbb{R}^n \longrightarrow \mathscr{A}$, $x \longmapsto f(x)$, with respect to the Lebesgue measure $\mu$. Hence, if $g \in L^2(\mathbb{R}^n, \mathscr{A})$, then
\begin{equation} \label{eq:L^2}
\|g\|_{L^2} := \left( \int_{\mathbb{R}^n} \|g(x)\|_\mathscr{A}^2 \, dx \right)^{1/2} = \left( \int_{\mathbb{R}^n} \|g(x)^* g(x)\|_\mathscr{A} \, dx \right)^{1/2} < + \infty.
\end{equation}
Comparing with \eqref{eq:E_n}, it is clear that $\|g\|_2 \leq \|g\|_{L^2} < + \infty$, for all $g \in \mathcal{S}^\mathscr{A}(\mathbb{R}^n)$. Also, $\mathcal{S}^\mathscr{A}(\mathbb{R}^n)$ is dense in $L^2(\mathbb{R}^n, \mathscr{A})$ with respect to the norm $\|\, \cdot \,\|_{L^2}$ (this follows from a standard argument combining \cite[Lemma 1.2.31, p.~29]{analysis-bochner} and \cite[Proposition 1.2.32, p.~29]{analysis-bochner}). Another fact which we shall frequently use throughout the paper is that $L^2(\mathbb{R}^n, \mathscr{A})$ is continuously embedded in $E_n$ as a dense subspace \cite[Appendix D]{cabralforgermelo}, with $L^2(\mathbb{R}^n)$ being isometrically embedded in $E_n$ \cite[Remark D.2, Appendix D]{cabralforgermelo}.

Let $J$ be a skew-symmetric linear transformation on $\mathbb{R}^n$, $f, g \in \mathcal{S}^\mathscr{A}(\mathbb{R}^n)$ and $x \in \mathbb{R}^n$. Then using the relations $\langle \xi, J\xi \rangle = \langle w, Jw \rangle = 0$, $\xi, w \in \mathbb{R}^n$ and Equation \eqref{rieffelop}, we obtain
\begin{align} \label{L1}
L_f(g)(x) &= \frac{1}{(2 \pi)^n} \int_{\mathbb{R}^n} \int_{\mathbb{R}^n} f \left( x - \frac{1}{2 \pi} J\xi \right) \, g(y) \, e^{i \langle \xi, x - y \rangle} \, dy \, d\xi \\ &= \frac{1}{(2 \pi)^{n / 2}} \int_{\mathbb{R}^n} f \left(x - \frac{1}{2 \pi} J\xi \right) \, \mathcal{F}(g)(\xi) \, e^{i \langle \xi, x \rangle} \, d\xi \nonumber \\ &= \frac{1}{(2 \pi)^n} \int_{\mathbb{R}^n} \int_{\mathbb{R}^n} \mathcal{F}(f)(w) \, \mathcal{F}(g)(\xi) \, e^{i \langle w + \xi, x - \frac{1}{2 \pi} J\xi \rangle} \, dw \, d\xi \nonumber \\ &= \frac{1}{(2 \pi)^n} \int_{\mathbb{R}^n} \int_{\mathbb{R}^n} \mathcal{F}(f)(w) \, \mathcal{F}(g)(\xi - w) \, e^{i \langle \xi, x - \frac{1}{2 \pi} J(\xi - w) \rangle} \, d\xi \, dw \nonumber \\ &= \frac{1}{(2 \pi)^n} \int_{\mathbb{R}^n} \int_{\mathbb{R}^n} \mathcal{F}(f)(w) \, \mathcal{F}(g)(\xi - w) \, e^{i \langle \xi, x + \frac{1}{2 \pi} Jw \rangle} \, d\xi \, dw \nonumber \\ &= \frac{1}{(2 \pi)^{n / 2}} \int_{\mathbb{R}^n} \mathcal{F}(f)(w) \, g \left( x + \frac{1}{2 \pi} Jw \right) \, e^{i \langle w, x \rangle} \, dw. \nonumber
\end{align}

Note that the integral in
\begin{equation*}
L_f(g)(\, \cdot \,) = \frac{1}{(2 \pi)^{n / 2}} \int_{\mathbb{R}^n} \mathcal{F}(f)(w) \, g \left( \, \cdot \, + \frac{1}{2 \pi} Jw \right) \, e^{i \langle w, \, \cdot \, \rangle} \, dw
\end{equation*}
remains absolutely convergent in the $L^2$-sense even if $g \in L^2(\mathbb{R}^n, \mathscr{A})$, with
\begin{align} \label{eq:l2-bounds}
\|L_f(g)\|_{L^2} & \leq \frac{1}{(2 \pi)^{n / 2}} \int_{\mathbb{R}^n} \|\mathcal{F}(f)(w) \, \mathcal{U}_{- \frac{1}{2 \pi} Jw, w, 0}(g)\|_{L^2} \, dw \\ & \leq \frac{1}{(2 \pi)^{n / 2}} \int_{\mathbb{R}^n} \|\mathcal{F}(f)(w)\|_\mathscr{A} \, \|\mathcal{U}_{- \frac{1}{2 \pi} Jw, w, 0}(g)\|_{L^2} \, dw \nonumber = \frac{1}{(2 \pi)^{n / 2}} \, \|\mathcal{F}(f)\|_1 \, \|g\|_{L^2}, \nonumber
\end{align}
where $\|\, \cdot \,\|_1$ denotes the usual L$^1$-norm. This proves the following useful result:

\begin{lemma} \label{lem:l2}
Let $f \in \mathcal{S}^\mathscr{A}(\mathbb{R}^n)$. Then $L_f$ extends to a continuous operator on $L^2(\mathbb{R}^n, \mathscr{A})$ such that the evaluation of $L_f$ on an element $g \in L^2(\mathbb{R}^n, \mathscr{A})$ is given by
\begin{equation} \label{eq:l2}
L_f(g)(\, \cdot \,) = \frac{1}{(2 \pi)^{n / 2}} \int_{\mathbb{R}^n} \mathcal{F}(f)(\xi) \, g \left( \, \cdot \, + \frac{1}{2 \pi} J\xi \right) \, e^{i \langle \xi, \, \cdot \, \rangle} \, d\xi, \qquad \text{with}
\end{equation}
\begin{equation} \label{eq:l2-estimates}
\|L_f(g)\|_{L^2} \leq \frac{1}{(2 \pi)^{n / 2}} \, \|\mathcal{F}(f)\|_1 \, \|g\|_{L^2}.
\end{equation}
\end{lemma}

Now consider a function $\psi \in C_c^\infty(\mathbb{R}^n)$ satisfying $\psi \geq 0$ and $\int_{\mathbb{R}^n} \psi(\xi) \, d\xi = 1$ such that its support $\text{supp }\psi$ is contained in the open ball $B(0, 1)$ of radius 1, centered at the origin. For each $m \in \mathbb{N}^*$, define $\psi_m(\xi) := m^n \, \psi(m \xi)$, so that $\text{supp }\psi_m \subseteq B(0, 1/m)$ and $\int_{\mathbb{R}^n} \psi_m(\xi) \, d\xi = 1$, and define
\begin{equation} \label{eq:e_m}
e_m := (2 \pi)^{n / 2} \, \mathcal{F}^{-1}(\psi_m), \qquad \tilde{e}_m := e_m \cdot 1_\mathscr{A}.
\end{equation}

\begin{lemma} \label{lem:polynomial}
For each fixed $g \in L^2(\mathbb{R}^n, \mathscr{A})$ and each polynomial function $p$ on $\mathbb{R}^n$ satisfying $p(0) = 0$, we have
\begin{equation} \label{eq:polynomial}
\lim_{m \rightarrow + \infty} \int_{\mathbb{R}^n} p(\xi) \, \psi_m(\xi) \, g \left( \, \cdot \, + \frac{1}{2 \pi} J\xi \right) e^{i \langle \xi, \, \cdot \, \rangle} \, d\xi = 0 \qquad \text{in $L^2(\mathbb{R}^n, \mathscr{A})$.}
\end{equation}
\end{lemma}

\begin{proof}
If $p(\xi) := \sum_{0 < |\alpha| \leq d} c_\alpha \, \xi^\alpha$, with $c_\alpha \in \mathbb{C}$, $\alpha \in \mathbb{N}^n$, then
\begin{equation*}
p(\xi) \, \mathcal{F}(\tilde{e}_m)(\xi) = \mathcal{F} \left( p \left( \frac{1}{i^{\alpha_1}} \, \frac{\partial^{\alpha_1}}{\partial \xi_1^{\alpha_1}}, \ldots, \frac{1}{i^{\alpha_n}} \, \frac{\partial^{\alpha_n}}{\partial \xi_n^{\alpha_n}} \right) \tilde{e}_m \right)(\xi), \qquad m \in \mathbb{N}^*.
\end{equation*}
Therefore, 
\begin{equation}
p(\xi) \, \psi_m(\xi) \cdot 1_\mathscr{A} = \mathcal{F} \left( \underbrace{\frac{1}{(2 \pi)^{n / 2}} \, p \left( \frac{1}{i^{\alpha_1}} \, \frac{\partial^{\alpha_1}}{\partial \xi_1^{\alpha_1}}, \ldots, \frac{1}{i^{\alpha_n}} \, \frac{\partial^{\alpha_n}}{\partial \xi_n^{\alpha_n}} \right) \tilde{e}_m}_{:= f_m} \right)(\xi), \qquad m \in \mathbb{N}^*,
\end{equation}
so, as noted in Lemma \ref{lem:l2}, the integral in Equation \eqref{eq:polynomial} is absolutely convergent in $L^2(\mathbb{R}^n, \mathscr{A})$ and equals $L_{f_m}(g)$. Fix $\epsilon > 0$ and let $\delta > 0$ be a real number such that $|p(\xi)| < \epsilon / (\|g(\, \cdot \,)\|_{L^2} + 1)$, for all $\xi \in \mathbb{R}^n$ satisfying $|\xi| < \delta$. Also, fix $m_0 \in \mathbb{N}^*$ such that $1 / m_0 < \delta$. Then reasoning similarly as in Equation \eqref{eq:l2-bounds} shows that
\begin{align}
& \left\| \int_{\mathbb{R}^n} p(\xi) \, \psi_m(\xi) \, g \left(\, \cdot \, + \frac{1}{2 \pi} J\xi \right) e^{i \langle \xi, \, \cdot \, \rangle} \, d\xi \right\|_{L^2} \leq \int_{\mathbb{R}^n} \psi_m(\xi) \, |p(\xi)| \|g(\, \cdot \,)\|_{L^2} \, d\xi \\ = & \underbrace{\int_{|\xi| < \delta} \psi_m(\xi) \, |p(\xi)| \|g(\, \cdot \,)\|_{L^2} \, d\xi}_{(I)} + \underbrace{\int_{|\xi| \geq \delta} \psi_m(\xi) \, |p(\xi)| \|g(\, \cdot \,)\|_{L^2} \, d\xi}_{(II)}, \qquad m \in \mathbb{N}^*. \nonumber
\end{align}
Since $\int_{\mathbb{R}^n} \psi_m(\xi) \, d\xi = 1$, for all $m \in \mathbb{N}^*$, we conclude that $(I) < \epsilon$. Furthermore, for all $m \in \mathbb{N}$ satisfying $m \geq m_0$, we have $(II) = 0$, since $\text{supp }\psi_m \subseteq B(0, 1 / m) \subseteq B(0, \delta)$. This proves the result.
\end{proof}

\begin{lemma} \label{lem:derivatives}
Let $D_0 := \sum_{0 < |\alpha| \leq d} c_\alpha \, \partial^\alpha$ be a constant coefficient differential operator of order $d$, where $c_\alpha \in \mathbb{C}$, $\alpha \in \mathbb{N}^{2n}$ and $\partial^\alpha := \partial_1^{\alpha_1} \ldots \partial_{2n}^{\alpha_{2n}}$ is a monomial in the generators of the adjoint representation $\text{Ad }\mathcal{U}$ (note that $c_0 = 0$). For each fixed $g \in L^2(\mathbb{R}^n, \mathscr{A})$, we have the equality
\begin{equation*}
\lim_{m \rightarrow + \infty} D_0(L_{\tilde{e}_m})(g) = 0 \qquad \text{in $L^2(\mathbb{R}^n, \mathscr{A})$.}
\end{equation*}
\end{lemma}

\begin{proof}
Denote by $(f_k)_{1 \leq k \leq n}$ the canonical basis of $\mathbb{R}^n$ and, for each $v \in \mathbb{R}^n$, denote by $\partial_v$ the operator on $\mathcal{S}^\mathscr{A}(\mathbb{R}^n)$ which associates the corresponding directional derivative $\partial_v \phi$ of a function $\phi \in \mathcal{S}^\mathscr{A}(\mathbb{R}^n)$. We note that, as shown in \cite[Subsection ``The algebra $\mathcal{B}_J^\mathscr{A}(\mathbb{R}^n)$'', Equation (3.22)]{cabralforgermelo}, the evaluation of the differential operator $D_0$ on $L_{\tilde{e}_m}$ is legitimate, with
\begin{equation*}
\partial_1^{\beta_1} \ldots \partial_{2n}^{\beta_{2n}} (L_{\tilde{e}_m}) = (-1)^{|\beta|} \, L_{\partial_{v_1}^{\beta_1} \ldots \partial_{v_{2n}}^{\beta_{2n}} \tilde{e}_m},
\end{equation*}
where $|\beta| := \sum_{k = 1}^{2n} \beta_k$, $v_k = f_k$, if $1 \leq k \leq n$ and $v_k = J(f_{k - n}) / 2 \pi$, if $n + 1 \leq k \leq 2n$ (note that $\partial_k$ is acting on operators in $C^\infty(\text{Ad }\mathcal{U})$, while $\partial_{v_k}$ acts on functions in $\mathcal{S}^\mathscr{A}(\mathbb{R}^n)$).

Let $T_v$ be the translation operator on $E_n$ defined in Lemma \ref{lem:s_to_cinfty}. Then since $L^2(\mathbb{R}^n, \mathscr{A})$ is continuously embedded in $E_n$ as a dense subspace, we may apply the operators $T_v$ to elements of $L^2(\mathbb{R}^n, \mathscr{A})$. In the following calculations, we will use the simplified notation $\mathcal{U}_{h f_k}$ to denote both the operators $\mathcal{U}_{h f_k, 0, 0}$ and $\mathcal{U}_{0, h f_k, 0}$. Using \eqref{eq:l2} and \eqref{eq:l2-estimates} we obtain, for a fixed $1 \leq k \leq 2n$ and every $0 \neq h \in \mathbb{R}$, $\phi \in \mathcal{S}^\mathscr{A}(\mathbb{R}^n)$:
\begin{align*}
& \left\| \left[ \frac{\mathcal{U}_{h f_k} \, L_\phi \, \mathcal{U}_{-h f_k} - L_\phi}{h} \right](g)(\, \cdot \,) - \frac{1}{(2 \pi)^{n / 2}} \int_{\mathbb{R}^n} \mathcal{F}(- \partial_{v_k} \phi)(\xi) \, g \left( \, \cdot \, + \frac{1}{2 \pi} J\xi \right) \, e^{i \langle \xi, \, \cdot \, \rangle} \, d\xi \right\|_{L^2} \\
& = \left\| \frac{1}{(2 \pi)^{n / 2}} \int_{\mathbb{R}^n} \mathcal{F} \left( \frac{T_{h v_k} \phi - \phi}{h} - (-\partial_{v_k}) \phi \right)(\xi) \, g \left( \, \cdot \, + \frac{1}{2 \pi} J\xi \right) \, e^{i \langle \xi, \, \cdot \, \rangle} \, d\xi \right\|_{L^2} \\
& \leq \frac{1}{(2 \pi)^{n / 2}} \, \left\| \mathcal{F} \left( \frac{T_{h v_k} \phi - \phi}{h} - (-\partial_{v_k}) \phi \right) \right\|_1 \, \|g\|_{L^2}.
\end{align*}
Since
\begin{equation*}
\mathcal{F} \left( \frac{T_{h v_k} \phi - \phi}{h} - (- \partial_{v_k}) \phi \right) \longrightarrow 0 \quad \text{in $\mathcal{S}^\mathscr{A}(\mathbb{R}^n)$ (hence, in $L^1(\mathbb{R}^n, \mathscr{A})$)}
\end{equation*}
when $h \rightarrow 0$, we see that
\begin{equation*}
\partial_k(L_\phi)(g) = - \frac{1}{(2 \pi)^{n / 2}} \int_{\mathbb{R}^n} \mathcal{F}(\partial_{v_k} \phi)(\xi) \, g \left( \, \cdot \, + \frac{1}{2 \pi} J\xi \right) \, e^{i \langle \xi, \, \cdot \, \rangle} \, d\xi \, \, (= - L_{\partial_{v_k} \phi}(g)), \qquad \phi \in \mathcal{S}^\mathscr{A}(\mathbb{R}^n).
\end{equation*}
Iterating this procedure and using the identity $\mathcal{F}((\partial/\partial_{\xi_l})(\phi))(\xi) = i \xi_l \, (\mathcal{F}(\phi))(\xi)$, $1 \leq l \leq n$, shows that, if $\alpha \in \mathbb{N}^{2n}$, $\alpha \neq 0$, then
\begin{equation*}
\partial^\alpha (L_\phi)(g) = \frac{1}{(2 \pi)^{n / 2}} \int_{\mathbb{R}^n} p_\alpha(\xi) \, \mathcal{F}(\phi)(\xi) \, g \left( \, \cdot \, + \frac{1}{2 \pi} J\xi \right) \, e^{i \langle \xi, \, \cdot \, \rangle} \, d\xi, \qquad \phi \in \mathcal{S}^\mathscr{A}(\mathbb{R}^n)
\end{equation*}
where $p_\alpha$ is a polynomial of degree $deg \, p_\alpha \geq 1$ satisfying $p_\alpha(0) = 0$. Substituting $\phi = \tilde{e}_m$ and using that $\mathcal{F}(\tilde{e}_m) = (2 \pi)^{n / 2} \, \psi_m \cdot 1_\mathscr{A}$ we get from the definition of $D_0$ the relation
\begin{equation*}
D_0(L_{\tilde{e}_m})(g) = \int_{\mathbb{R}^n} p(\xi) \, \psi_m(\xi) \, g \left( \, \cdot \, + \frac{1}{2 \pi} J\xi \right) \, e^{i \langle \xi, \, \cdot \, \rangle} \, d\xi, \qquad m \in \mathbb{N}^*,
\end{equation*}
where $p$ is a linear combination of polynomials $q$ of degree $deg \, q \geq 1$ satisfying $q(0) = 0$ (so $p$ also satisfies $p(0) = 0$). Therefore, by Lemma \ref{lem:polynomial}, the limit $\lim_{m \rightarrow + \infty} D_0(L_{\tilde{e}_m})(g)$ exists in $L^2(\mathbb{R}^n, \mathscr{A})$ and equals zero.
\end{proof}

Now consider the symbol map (see \cite[Equation (12)]{melomerklen2})
\begin{equation*}
S \colon C^\infty(\text{Ad }\mathcal{U}) \longrightarrow \mathcal{B}^\mathscr{A}(\mathbb{R}^{2n})
\end{equation*}
given by
\begin{equation} \label{S}
S(A)(x, \xi) := (2 \pi)^{n/2} \langle u \cdot 1_\mathscr{A}, \left\{ (D \, [(\text{Ad }\mathcal{U})_{-x, -\xi}(A)] \circ \mathcal{F}^{-1}) \otimes I_{E_n} \right\} \, v \cdot 1_\mathscr{A} \rangle_{E_{2n}},
\end{equation}
for all $A \in C^\infty(\text{Ad }\mathcal{U})$ and $(x, \xi) \in \mathbb{R}^{2n}$, where $D := \prod_{j = 1}^n (1 + \partial_{x_j})^2 (1 + \partial_{\xi_j})^2$ and $u$ and $v$ are (fixed) suitable scalar-valued functions belonging to $L^2(\mathbb{R}^{2n}) \cap L^1(\mathbb{R}^{2n})$ (for more information on the tensor product operator in Equation \eqref{S}, see \cite[Appendix C, Lemma C.1]{cabralforgermelo}). We make the trivial, but important, observation that the formula defining the map $S$ remains unchanged if we substitute $\left\{ (D \, [(\text{Ad }\mathcal{U})_{-x, -\xi}(A)] \circ \mathcal{F}^{-1}) \otimes I_{E_n} \right\}$ by the restricted map
\begin{equation*}
\left\{ (D \, [(\text{Ad }\mathcal{U})_{-x, -\xi}(A)] \circ \mathcal{F}^{-1}) \otimes I_{E_n} \right\}|_{L^2(\mathbb{R}^{2n}) \cdot 1_\mathscr{A}} \colon L^2(\mathbb{R}^{2n}) \cdot 1_\mathscr{A} \longrightarrow E_{2n}.
\end{equation*}
Define the restriction map $\mathcal{R} \colon \mathcal{B}^\mathscr{A}(\mathbb{R}^{2n}) \longrightarrow \mathcal{B}^\mathscr{A}(\mathbb{R}^n)$ which takes the function $f \colon (x, \xi) \longmapsto f(x, \xi)$ in $\mathcal{B}^\mathscr{A}(\mathbb{R}^{2n})$ to the function $\mathcal{R}f(x) := f(x, 0)$, and let
\begin{equation*}
L \colon \mathcal{B}^\mathscr{A}(\mathbb{R}^n) \longrightarrow C^\infty(\text{Ad }\mathcal{U})
\end{equation*}
be the map which sends a function $f \in \mathcal{B}^\mathscr{A}(\mathbb{R}^n)$ to the operator $L_f$. Using \cite[Theorem 1]{melomerklen2} (which remains valid even if $\mathscr{A}$ is non-separable -- see \cite[Appendix C]{cabralforgermelo}), we see that this map satisfies
\begin{equation} \label{eq:left_inverse}
\mathcal{R} \circ S \circ L = Id_{\mathcal{B}^\mathscr{A}(\mathbb{R}^n)},
\end{equation}
where $\text{Id}_{\mathcal{B}^\mathscr{A}(\mathbb{R}^n)}$ is the identity operator on $\mathcal{B}^\mathscr{A}(\mathbb{R}^n)$.

Before we can establish the main result of this section (Proposition \ref{prop:aprox}), we need two auxiliary lemmas (the first one, below, is an adaptation of \cite[Proposi\c c\~ao 2.5]{merklentese}):

\begin{lemma} \label{lem:approxidentity}
Let $(\tilde{e}_m)_{m \in \mathbb{N}^*}$ be the sequence in $\mathcal{S}^\mathscr{A}(\mathbb{R}^n)$ introduced in Equation \eqref{eq:e_m}. Then for every $g \in \mathcal{S}^\mathscr{A}(\mathbb{R}^n)$, we have the equality
\begin{equation} \label{eq:approxidentity0}
\lim_{m \rightarrow + \infty} L_{\tilde{e}_m}(g) = g \qquad \text{in $E_n$.}
\end{equation}
\end{lemma}

\begin{proof}
Fix a skew-symmetric linear transformation $J$ on $\mathbb{R}^n$ and $0 \neq g \in \mathcal{S}^\mathscr{A}(\mathbb{R}^n)$. We begin the proof by showing that, for any $\epsilon > 0$ there exists $\delta > 0$ with the property that, for all $x \in \mathbb{R}^n$ and all $\xi \in \mathbb{R}^n$ satisfying $|\xi| < \delta$, we have
\begin{equation} \label{eq:approxidentity}
\|(1 + |x|^2)^{n / 2} \, [e^{i \langle x, \xi \rangle} \, g(x + J\xi) - g(x)]\|_\mathscr{A} < \epsilon.
\end{equation}

Take $\epsilon > 0$ and define $K := \|g\|_\infty := \sup_{y \in \mathbb{R}^n} \|g(y)\|_\mathscr{A}$. Since $g$ belongs to $\mathcal{S}^\mathscr{A}(\mathbb{R}^n)$ there exists $R_1 > 1$ such that
\begin{equation} \label{eq:approxidentity1}
(1 + |y|^2)^{n / 2} \, \|g(y)\|_\mathscr{A} \leq \epsilon / (1 + 4^{n / 2}), \qquad \text{whenever $|y| > R_1$}.
\end{equation}
Moreover, choose $R_2 > 0$ such that $|J\xi| < 1$, for all $\xi \in \mathbb{R}^n$ satisfying $|\xi| < R_2$.

By the uniform continuity of the maps
\begin{equation*}
(x, \xi) \longmapsto (1 + |x + J\xi|^2)^{n / 2} \qquad \text{and} \qquad (x, \xi) \longmapsto (1 + |x + J\xi|^2)^{n / 2} \, e^{i \langle x, \xi \rangle} \, g(x + J\xi)
\end{equation*}
on the compact set $S := \left\{ (x, \xi) \in \mathbb{R}^{2n}: |x| \leq R_1 + 1, |\xi| \leq R_2 \right\}$, we may find $0 < R_3 < R_2$ such that, if $(x, \xi) \in S$ and $|(x, \xi) - (x, 0)| = |\xi| < R_3$, then
\begin{equation*}
|(1 + |x + J\xi|^2)^{n / 2} - (1 + |x|^2)^{n / 2}| < \epsilon / (2K)
\end{equation*}
and
\begin{equation*}
\|(1 + |x + J\xi|^2)^{n / 2} \, e^{i \langle x, \xi \rangle} \, g(x + J\xi) - (1 + |x|^2)^{n / 2} \, g(x)\|_\mathscr{A} < \epsilon / 2.
\end{equation*}
Hence, if $|x| \leq R_1 + 1$ and $|\xi| < R_3$, we have
\begin{align*}
& \|(1 + |x|^2)^{n / 2} \, e^{i \langle x, \xi \rangle} \, g(x + J\xi) - (1 + |x|^2)^{n / 2} \, g(x)\|_\mathscr{A} \\ & \leq \|(1 + |x + J\xi|^2)^{n / 2} \, e^{i \langle x, \xi \rangle} \, g(x + J\xi) - (1 + |x|^2)^{n / 2} \, g(x)\|_\mathscr{A} \\ & + |(1 + |x + J\xi|^2)^{n / 2} - (1 + |x|^2)^{n / 2}| \, \|g(x + J\xi)\|_\mathscr{A} < \epsilon / 2 + \epsilon / 2 = \epsilon.
\end{align*}
On the other hand, if $|x| > R_1 + 1$ and $|\xi| < R_3$, then $|x + J\xi| \geq |x| - |J\xi| > R_1 > 1$, so $|J\xi| < 1 < |x + J\xi|$. Therefore, $|x| \leq |x + J\xi| + |J\xi| < 2 |x + J\xi|$ which, when combined with \eqref{eq:approxidentity1}, implies the estimates
\begin{align*}
& \|(1 + |x|^2)^{n / 2} \, e^{i \langle x, \xi \rangle} \, g(x + J\xi) - (1 + |x|^2)^{n / 2} \, g(x)\|_\mathscr{A} \\ & \leq (1 + 4 |x + J\xi|^2)^{n / 2} \, \|g(x + J\xi)\|_\mathscr{A} + (1 + |x|^2)^{n / 2} \, \|g(x)\|_\mathscr{A} \\ & \leq 4^{n / 2} \, (1 + |x + J\xi|^2)^{n / 2} \, \|g(x + J\xi)\|_\mathscr{A} + (1 + |x|^2)^{n / 2} \, \|g(x)\|_\mathscr{A} \\ & \leq (1 + 4^{n / 2}) \sup_{|y| > R_1} [(1 + |y|^2)^{n / 2} \, \|g(y)\|_\mathscr{A}] \leq (1 + 4^{n / 2}) \cdot \frac{\epsilon}{1 + 4^{n / 2}} = \epsilon
\end{align*}
(note that we have used the fact that $|x + J\xi| > R_1$). This establishes \eqref{eq:approxidentity} with $\delta := R_3$.

Now, we apply what was just proved for the skew-symmetric linear transformation $J' = J / (2 \pi)$: if we fix $\epsilon_0 > 0$ then we may use the relation (see Equation \eqref{L1})
\begin{equation*}
L_f(g)(x) = \frac{1}{(2 \pi)^{n/2}} \int_{\mathbb{R}^n} \mathcal{F}(f)(w) \, g \left( x + \frac{1}{2 \pi} Jw \right) \, e^{i \langle w, x \rangle} \, dw, \qquad f \in \mathcal{S}^\mathscr{A}(\mathbb{R}^n), \, x \in \mathbb{R}^n
\end{equation*}
to obtain $\delta_0 > 0$ such that, for all $m \in \mathbb{N}$ satisfying $m > 1 / \delta_0$ and any fixed $x \in \mathbb{R}^n$,
\begin{equation*}
(1 + |x|^2)^{n / 2} \, \|L_{\tilde{e}_m}(g)(x) - g(x)\|_\mathscr{A} \leq \int_{\mathbb{R}^n} \psi_m(\xi) \, (1 + |x|^2)^{n / 2} \, \left\| e^{i \langle x, \xi \rangle} g \left( x + \frac{1}{2 \pi} J\xi \right) - g(x) \right\|_\mathscr{A} \, d\xi < \epsilon_0.
\end{equation*}
Define $L_0 := \int_{\mathbb{R}^n} (1 + |x|^2)^{- n} \, dx$. Then the above inequality allows us to conclude, for all $m \in \mathbb{N}$ satisfying $m > 1 / \delta_0$, the estimates
\begin{align}
\|L_{\tilde{e}_m}(g) - g\|_2^2 & \leq \int_{\mathbb{R}^n} \|L_{\tilde{e}_m}(g)(x) - g(x)\|_\mathscr{A}^2 \, dx \\ & = \int_{\mathbb{R}^n} (1 + |x|^2)^{- n} \, [(1 + |x|^2)^{n / 2} \, \|L_{\tilde{e}_m}(g)(x) - g(x)\|_\mathscr{A}]^2 \, dx < L_0 \, \epsilon_0^2. \nonumber
\end{align}
This establishes \eqref{eq:approxidentity0}.
\end{proof}

\begin{corollary} \label{cor:approxidentity}
Let $(\tilde{e}_m)_{m \in \mathbb{N}^*}$ be the sequence in $\mathcal{S}^\mathscr{A}(\mathbb{R}^n)$ introduced in Equation \eqref{eq:e_m}. Then for every $\tilde{f} \in L^2(\mathbb{R}^n) \cdot 1_\mathscr{A}$, we have the equality
\begin{equation} \label{eq:approxidentity_corollary}
\lim_{m \rightarrow + \infty} L_{\tilde{e}_m}(\tilde{f}) = \tilde{f} \qquad \text{in $E_n$.}
\end{equation}
\end{corollary}

\begin{proof}
Using Equation \eqref{eq:l2-estimates} of Lemma \ref{lem:l2} and the definition of $\tilde{e}_m := (2 \pi)^{n / 2} \, \mathcal{F}^{-1}(\psi_m) \cdot 1_\mathscr{A}$ gives
\begin{equation*}
\|L_{\tilde{e}_m}(\tilde{h})\|_{L^2} \leq \|\psi_m \cdot 1_\mathscr{A}\|_1 \, \|\tilde{h}\|_{L^2} = \|\tilde{h}\|_{L^2}, \qquad m \in \mathbb{N}^*, \, \tilde{h} \in L^2(\mathbb{R}^n) \cdot 1_\mathscr{A}.
\end{equation*}
Therefore, the result follows from Lemma \ref{lem:approxidentity} by noting that (1) each $L_{\tilde{e}_m}$ leaves $L^2(\mathbb{R}^n) \cdot 1_\mathscr{A}$ invariant, (2) restricting the topologies of $L^2(\mathbb{R}^n, \mathscr{A})$ and of $E_n$ to $L^2(\mathbb{R}^n) \cdot 1_\mathscr{A}$ yield the same canonical topology, (3) $\mathcal{S}(\mathbb{R}^n) \cdot 1_\mathscr{A}$ is dense in $L^2(\mathbb{R}^n) \cdot 1_\mathscr{A}$ with respect to this topology and that (4) the estimates
\begin{equation*}
\|L_{\tilde{e}_m}(\tilde{f}) - \tilde{f}\|_{L^2} \leq \|L_{\tilde{e}_m}(\tilde{f} - \tilde{g})\|_{L^2} + \|L_{\tilde{e}_m}(\tilde{g}) - \tilde{g}\|_{L^2} + \|\tilde{g} - \tilde{f}\|_{L^2}, \qquad m \in \mathbb{N}^*
\end{equation*}
hold, for every $\tilde{f} \in L^2(\mathbb{R}^n) \cdot 1_\mathscr{A}$ and $\tilde{g} \in \mathcal{S}(\mathbb{R}^n) \cdot 1_\mathscr{A}$.
\end{proof}

\begin{lemma} \label{prop:aprox'}
For every $A \in C^\infty(\text{Ad }\mathcal{U})$, $\tilde{g} \in L^2(\mathbb{R}^n) \cdot 1_\mathscr{A}$ and $(x, \xi) \in \mathbb{R}^{2n}$, we have
\begin{equation} \label{eq:pointwiseconvergence}
\lim_{m \rightarrow + \infty} \left\{ D \, [(\text{Ad }\mathcal{U})_{-x, -\xi}(A \circ L_{\tilde{e}_m} - A)] \circ \mathcal{F}^{-1} \right\}(\tilde{g}) = 0 \qquad \text{in $E_n$.}
\end{equation}
\end{lemma}

\begin{proof}
First, note that
\begin{align} \label{eq:linearcombination}
& D \, [(\text{Ad }\mathcal{U})_{-x, -\xi}(A \circ L_{\tilde{e}_m})] \circ \mathcal{F}^{-1} = \left\{ [D \, ((\text{Ad }\mathcal{U})_{-x, -\xi}(A))] \circ L_{\tilde{e}_m} \right\} \circ \mathcal{F}^{-1} \\
& + \text{ linear combination of terms of the form } \nonumber \\
& \left\{ \partial_{x, \xi}^\alpha [(\text{Ad }\mathcal{U})_{-x, -\xi}(A)] \circ [\mathcal{U}_{-x, -\xi} \, \partial^\beta \, (L_{\tilde{e}_m}) \, (\mathcal{U}_{-x, -\xi})^{-1}] \right\} \circ \mathcal{F}^{-1}, \nonumber
\end{align}
where the $\partial^\beta$'s, $\beta \neq 0$, are monomials in the generators $(\partial_k)_{1 \leq k \leq 2n}$ of the adjoint representation $\text{Ad }\mathcal{U}$. Therefore, for all $\tilde{g} \in L^2(\mathbb{R}^n) \cdot 1_\mathscr{A}$, we have as a consequence of Lemma \ref{lem:derivatives} and Corollary \ref{cor:approxidentity} that
\begin{align} \label{eq:pointwiseconvergence2}
& \lim_{m \rightarrow + \infty} \left\{ D \, [(\text{Ad }\mathcal{U})_{-x, -\xi}(A \circ L_{\tilde{e}_m})] \circ \mathcal{F}^{-1} \right\}(\tilde{g}) \\
= & \lim_{m \rightarrow + \infty} \left\{ [D \, ((\text{Ad }\mathcal{U})_{-x, -\xi}(A))] \circ L_{\tilde{e}_m} \circ \mathcal{F}^{-1} \right\}(\tilde{g}) \nonumber \\
= & \left\{ [D \, ((\text{Ad }\mathcal{U})_{-x, -\xi}(A))] \circ \mathcal{F}^{-1} \right\}(\tilde{g}), \nonumber
\end{align}
where the limit is performed in the Hilbert C$^*$-module $E_n$ (for the nonzero order terms, we have used that the $L^2$-topology is finer than the topology of $E_n$). This proves Equation \eqref{eq:pointwiseconvergence}.
\end{proof}

\begin{proposition} \label{prop:aprox}
For every $A \in C^\infty(\text{Ad }\mathcal{U})$ and each fixed $(x, \xi) \in \mathbb{R}^{2n}$, we have
\begin{equation} \label{eq:aprox_limit}
\lim_{m \rightarrow + \infty} S(A \circ L_{\tilde{e}_m})(x, \xi) = S(A)(x, \xi) \qquad \text{in $\mathscr{A}$.}
\end{equation}
\end{proposition}

\begin{proof}
Let us begin by majorizing the norms of the (restricted) linear maps (from $L^2(\mathbb{R}^{2n}) \cdot 1_\mathscr{A}$ to $E_{2n}$)
\begin{equation*}
\left\{ (D \, [(\text{Ad }\mathcal{U})_{-x, -\xi}(A \circ L_{\tilde{e}_m})] \circ \mathcal{F}^{-1}) \otimes I_{E_n} \right\}|_{L^2(\mathbb{R}^{2n}) \cdot 1_\mathscr{A}}, \qquad m \in \mathbb{N}^*, \, (x, \xi) \in \mathbb{R}^{2n},
\end{equation*}
uniformly in $m$ and in $(x, \xi)$, by establishing norm bounds on each of the summands in \eqref{eq:linearcombination}.

Using again, just as in Lemma \ref{lem:derivatives}, the fact that $\partial_1^{\beta_1} \ldots \partial_{2n}^{\beta_{2n}} (L_{\tilde{e}_m}) = (-1)^{|\beta|} \, L_{\partial_{v_1}^{\beta_1} \ldots \partial_{v_{2n}}^{\beta_{2n}} \tilde{e}_m}$, where $|\beta| := \sum_{k = 1}^{2n} \beta_k$, $v_k = f_k$, if $1 \leq k \leq n$ and $v_k = J(f_{k - n}) / 2 \pi$, if $n + 1 \leq k \leq 2n$, we see, in particular, that each $\partial^\beta \, (L_{\tilde{e}_m})$ sends $L^2(\mathbb{R}^n) \cdot 1_\mathscr{A}$ into itself, as a result of Lemma \ref{lem:l2}. This conclusion will be useful soon, when we have to write a certain operator (restricted to $L^2(\mathbb{R}^{2n}) \cdot 1_\mathscr{A}$) as the composition of two tensor product operators.

Applying Lemma \ref{lem:derivatives} and the Uniform Boundedness Principle for the restricted bounded operators $[\partial^\beta \, (L_{\tilde{e}_m})]|_{L^2(\mathbb{R}^n) \cdot 1_\mathscr{A}}$, $\beta \neq 0$, we conclude that there exists $M_\beta > 0$ such that 
\begin{equation*}
\sup_{m \in \mathbb{N}^*} \|[\partial^\beta \, (L_{\tilde{e}_m})]|_{L^2(\mathbb{R}^n) \cdot 1_\mathscr{A}}\| \leq M_\beta,
\end{equation*}
so
\begin{equation} \label{eq:unifboundednessprinciple}
\sup_{m \in \mathbb{N}^*} \|\left\{ [\mathcal{U}_{-x, -\xi} \, \partial^\beta \, (L_{\tilde{e}_m}) \, (\mathcal{U}_{-x, -\xi})^{-1}] \circ \mathcal{F}^{-1} \right\}|_{L^2(\mathbb{R}^n) \cdot 1_\mathscr{A}}\| \leq M_\beta
\end{equation}
(the norm, above, is the usual operator norm on $L^2(\mathbb{R}^n) \cdot 1_\mathscr{A}$; note that the constant $M_\beta$ does not depend on the fixed $(x, \xi) \in \mathbb{R}^{2n}$). Hence, the norm of each operator
\begin{equation*}
\left\{ [\mathcal{U}_{-x, -\xi} \, \partial^\beta \, (L_{\tilde{e}_m}) \, (\mathcal{U}_{-x, -\xi})^{-1}] \circ \mathcal{F}^{-1} \right\}|_{L^2(\mathbb{R}^n) \cdot 1_\mathscr{A}} \otimes I_{L^2(\mathbb{R}^n) \cdot 1_\mathscr{A}}
\end{equation*}
on the $L^2$-completion
\begin{equation*}
\left[ L^2(\mathbb{R}^n) \cdot 1_\mathscr{A} \right] \otimes \left[ L^2(\mathbb{R}^n) \cdot 1_\mathscr{A} \right] \simeq L^2(\mathbb{R}^{2n}) \cdot 1_\mathscr{A},
\end{equation*}
of the algebraic tensor product $\left[ L^2(\mathbb{R}^n) \cdot 1_\mathscr{A} \right] \otimes_{\text{alg}} \left[ L^2(\mathbb{R}^n) \cdot 1_\mathscr{A} \right]$ is bounded by $M_\beta$, independently of $m$ and of $(x, \xi)$. Combining this observation with the identity
\begin{align*}
& \left\{ \left\{ \partial_{x, \xi}^\alpha [(\text{Ad }\mathcal{U})_{-x, -\xi}(A)] \circ [\mathcal{U}_{-x, -\xi} \, \partial^\beta \, (L_{\tilde{e}_m}) \, (\mathcal{U}_{-x, -\xi})^{-1}] \circ \mathcal{F}^{-1} \right\} \otimes I_{E_n} \right\}|_{L^2(\mathbb{R}^{2n}) \cdot 1_\mathscr{A}} \\
& = \left\{ \partial_{x, \xi}^\alpha [(\text{Ad }\mathcal{U})_{-x, -\xi}(A)] \otimes I_{E_n} \right\}|_{L^2(\mathbb{R}^{2n}) \cdot 1_\mathscr{A}} \\
& \hspace{4.0cm} \circ \left\{ \left\{ [\mathcal{U}_{-x, -\xi} \, \partial^\beta \, (L_{\tilde{e}_m}) \, (\mathcal{U}_{-x, -\xi})^{-1}] \circ \mathcal{F}^{-1} \right\} |_{L^2(\mathbb{R}^n) \cdot 1_\mathscr{A}} \otimes I_{L^2(\mathbb{R}^n) \cdot 1_\mathscr{A}} \right\}
\end{align*}
we obtain for each term (corresponding to $\beta \neq 0$) in the linear combination appearing in \eqref{eq:linearcombination} the estimates
\begin{align*}
& \left\| \left\{ \left\{ \partial_{x, \xi}^\alpha [(\text{Ad }\mathcal{U})_{-x, -\xi}(A)] \circ [\mathcal{U}_{-x, -\xi} \, \partial^\beta \, (L_{\tilde{e}_m}) \, (\mathcal{U}_{-x, -\xi})^{-1}] \circ \mathcal{F}^{-1} \right\} \otimes I_{E_n} \right\}|_{L^2(\mathbb{R}^{2n}) \cdot 1_\mathscr{A}} \right\| \\
& \leq \left\| \partial_{x, \xi}^\alpha [(\text{Ad }\mathcal{U})_{-x, -\xi}(A)] \otimes I_{E_n} \right\| \\
& \hspace{8.6em} \cdot \left\| \left\{ [\mathcal{U}_{-x, -\xi} \, \partial^\beta \, (L_{\tilde{e}_m}) \, (\mathcal{U}_{-x, -\xi})^{-1}] \circ \mathcal{F}^{-1} \right\} |_{L^2(\mathbb{R}^n) \cdot 1_\mathscr{A}} \otimes I_{L^2(\mathbb{R}^n) \cdot 1_\mathscr{A}} \right\| \\
& \leq M_\beta \, \left\| \partial_{x, \xi}^\alpha [(\text{Ad }\mathcal{U})_{-x, -\xi}(A)] \right\| = M_\beta \, \left\| \mathcal{U}_{-x, -\xi} \, \partial^\alpha (A) \, (\mathcal{U}_{-x, -\xi})^{-1} \right\| \leq M_\beta \left\| \partial^\alpha (A) \right\|,
\end{align*}
where: (1) the norm in the first line is just the usual one of a bounded linear map from $L^2(\mathbb{R}^{2n}) \cdot 1_\mathscr{A}$ to $E_{2n}$; (2') the second norm in the second line (from left to right) is the usual operator norm on $L^2(\mathbb{R}^{2n}) \cdot 1_\mathscr{A}$; (2'') the first norm in the second line (from left to right) is the operator norm on $E_{2n}$ (we have implicitly used that, if $T \colon E_{2n} \longrightarrow E_{2n}$ is a bounded operator, if $\|T|_{L^2(\mathbb{R}^{2n}) \cdot 1_\mathscr{A}}\|_1$ denotes the usual norm of $T|_{L^2(\mathbb{R}^{2n}) \cdot 1_\mathscr{A}}$ as a bounded linear map from $L^2(\mathbb{R}^{2n}) \cdot 1_\mathscr{A}$ to $E_{2n}$, and $\|T\|_2$ denotes the usual operator norm on $E_{2n}$, then $\|T|_{L^2(\mathbb{R}^{2n}) \cdot 1_\mathscr{A}}\|_1 \leq \|T\|_2$); (3) the norms in the third line are all operator norms on $E_{2n}$.

Note that, in order to obtain the above identity, we have used the fact that
\begin{equation*}
[\mathcal{U}_{-x, -\xi} \, \partial^\beta \, (L_{\tilde{e}_m}) \, (\mathcal{U}_{-x, -\xi})^{-1}] \circ \mathcal{F}^{-1}
\end{equation*}
leaves $L^2(\mathbb{R}^n) \cdot 1_\mathscr{A}$ invariant, so the equality
\begin{align*}
& \left\{ \left\{ [\mathcal{U}_{-x, -\xi} \, \partial^\beta \, (L_{\tilde{e}_m}) \, (\mathcal{U}_{-x, -\xi})^{-1}] \circ \mathcal{F}^{-1} \right\} \otimes I_{E_n} \right\}|_{L^2(\mathbb{R}^{2n}) \cdot 1_\mathscr{A}} \\
& = \left\{ \left\{ [\mathcal{U}_{-x, -\xi} \, \partial^\beta \, (L_{\tilde{e}_m}) \, (\mathcal{U}_{-x, -\xi})^{-1}] \circ \mathcal{F}^{-1} \right\}|_{L^2(\mathbb{R}^n) \cdot 1_\mathscr{A}} \right\} \otimes I_{L^2(\mathbb{R}^n) \cdot 1_\mathscr{A}}
\end{align*}
holds. We also make the observation that one of the tensor products is performed between adjointable operators on $E_n$, while the other one is performed between bounded operators on $L^2(\mathbb{R}^n) \cdot 1_\mathscr{A}$ (both of them are denoted simply by ``$\otimes$'').

Finally, to deal with the first summand in Equation \eqref{eq:linearcombination} we note that, defining $\tilde{D} := \prod_{j = 1}^n (1 + \partial_j)^2 (1 + \partial_{j + n})^2$, we have the equality
\begin{equation*}
D \, [(\text{Ad }\mathcal{U})_{-x, -\xi}(A)] = (\text{Ad }\mathcal{U})_{-x, -\xi}(\tilde{D}(A)), \qquad (x, \xi) \in \mathbb{R}^{2n},
\end{equation*}
so adapting the argument contained in (2''), above, to $E_n$, we obtain the estimate
\begin{align*}
\| \left\{ [D \, ((\text{Ad }\mathcal{U})_{-x, -\xi}(A))] \circ L_{\tilde{e}_m} \circ \mathcal{F}^{-1} \right\}|_{L^2(\mathbb{R}^n) \cdot 1_\mathscr{A}} \| \leq \|\tilde{D}(A)\| \, \|L_{\tilde{e}_m}|_{L^2(\mathbb{R}^n) \cdot 1_\mathscr{A}}\| \, \|\mathcal{F}^{-1}|_{L^2(\mathbb{R}^n) \cdot 1_\mathscr{A}}\|,
\end{align*}
where (i) $\|\tilde{D}(A)\|$ is the operator norm of $\mathcal{L}_\mathscr{A}(E_n)$ evaluated on $\tilde{D}(A)$, and (ii) the other two norms involved are just the usual ones of bounded linear operators on $L^2(\mathbb{R}^n) \cdot 1_\mathscr{A}$. Therefore, since as a consequence of the first estimate obtained in the proof of Corollary \ref{cor:approxidentity}, we have $\|L_{\tilde{e}_m}|_{L^2(\mathbb{R}^n) \cdot 1_\mathscr{A}}\| \leq 1$, for all $m \in \mathbb{N}$, we get
\begin{equation*}
\| \left\{ [D \, ((\text{Ad }\mathcal{U})_{-x, -\xi}(A))] \circ L_{\tilde{e}_m} \circ \mathcal{F}^{-1} \right\}|_{L^2(\mathbb{R}^n) \cdot 1_\mathscr{A}} \| \leq \|\tilde{D}(A)\|.
\end{equation*}
In order to finish the proof of the desired uniform boundedness for the norms of
\begin{equation*}
\left\{ ([D \, ((\text{Ad }\mathcal{U})_{-x, -\xi}(A))] \circ L_{\tilde{e}_m} \circ \mathcal{F}^{-1}) \otimes I_{E_n} \right\}|_{L^2(\mathbb{R}^{2n}) \cdot 1_\mathscr{A}},
\end{equation*}
we need to ``tensor product'' this last estimate and go from dimension $n$ to $2n$. This can be done exactly as we did for the terms depending on a $\beta \neq 0$.

Therefore, we have just proved the existence of a constant $M > 0$, independent of $(x, \xi) \in \mathbb{R}^{2n}$, such that
\begin{equation} \label{eq:unifboundedness}
\sup_{m \in \mathbb{N}^*} \left\| \left\{ (D \, [(\text{Ad }\mathcal{U})_{-x, -\xi}(A \circ L_{\tilde{e}_m})] \circ \mathcal{F}^{-1}) \otimes I_{E_n} \right\}|_{L^2(\mathbb{R}^{2n}) \cdot 1_\mathscr{A}} \right\| \leq M.
\end{equation} 
 
Now we will show how to prove Equation \eqref{eq:aprox_limit} using \eqref{eq:pointwiseconvergence} and \eqref{eq:unifboundedness}. By the Cauchy-Schwarz inequality applied to Equation \eqref{S} (with $A$ substituted by $A \circ L_{\tilde{e}_m} - A$),
\begin{equation*}
\|S(A \circ L_{\tilde{e}_m} - A)(x, \xi)\|_\mathscr{A}
\end{equation*}
\begin{equation*}
\leq (2 \pi)^{n/2} \, \|u \cdot 1_\mathscr{A}\|_{E_{2n}} \, \|\left\{ (D \, [(\text{Ad }\mathcal{U})_{-x, -\xi}(A \circ L_{\tilde{e}_m} - A)] \circ \mathcal{F}^{-1}) \otimes I_{E_n} \right\} \, v \cdot 1_\mathscr{A} \|_{E_{2n}},
\end{equation*}
hence it suffices to show that
\begin{equation*}
\lim_{m \rightarrow + \infty} \|\left\{ (D \, [(\text{Ad }\mathcal{U})_{-x, -\xi}(A \circ L_{\tilde{e}_m} - A)] \circ \mathcal{F}^{-1}) \otimes I_{E_n} \right\} \, v \cdot 1_\mathscr{A} \|_{E_{2n}} = 0.
\end{equation*}
Fix $\epsilon > 0$ and define $K := \|\tilde{D}(A)\| = \|(\text{Ad }\mathcal{U})_{-x, -\xi}(\tilde{D}(A))\| = \|D \, [(\text{Ad }\mathcal{U})_{-x, -\xi}(A)]\|$, $(x, \xi) \in \mathbb{R}^{2n}$. Since the algebraic tensor product $L^2(\mathbb{R}^n) \otimes_{\text{alg}} L^2(\mathbb{R}^n)$ can be naturally identified as a dense subspace of $L^2(\mathbb{R}^{2n})$, we can find $f \in L^2(\mathbb{R}^n) \otimes_{\text{alg}} L^2(\mathbb{R}^n)$ such that
\begin{equation*}
\|(f - v) \cdot 1_{\mathscr{A}}\|_{L^2} < \frac{\epsilon}{3 (M + 1) (K + 1)}.
\end{equation*}
Moreover, by \eqref{eq:pointwiseconvergence} there exists $m_0 \in \mathbb{N}$ such that $m \geq m_0$ implies
\begin{equation*}
\|\left\{ (D \, [(\text{Ad }\mathcal{U})_{-x, -\xi}(A \circ L_{\tilde{e}_m} - A)] \circ \mathcal{F}^{-1}) \otimes I_{E_n} \right\} \, f \cdot 1_\mathscr{A} \|_{E_{2n}} < \frac{\epsilon}{3} \cdot
\end{equation*}
Therefore, combining these approximations with \eqref{eq:unifboundedness} yields for every $m \geq m_0$ the estimates
\begin{align*}
& \|\left\{ (D \, [(\text{Ad }\mathcal{U})_{-x, -\xi}(A \circ L_{\tilde{e}_m} - A)] \circ \mathcal{F}^{-1}) \otimes I_{E_n} \right\} \, v \cdot 1_\mathscr{A} \|_{E_{2n}} \\
\leq & \|\left\{ (D \, [(\text{Ad }\mathcal{U})_{-x, -\xi}(A \circ L_{\tilde{e}_m})] \circ \mathcal{F}^{-1}) \otimes I_{E_n} \right\} \, (f - v) \cdot 1_\mathscr{A} \|_{E_{2n}} \\
+ & \|\left\{ (D \, [(\text{Ad }\mathcal{U})_{-x, -\xi}(A \circ L_{\tilde{e}_m} - A)] \circ \mathcal{F}^{-1}) \otimes I_{E_n} \right\} \, f \cdot 1_\mathscr{A} \|_{E_{2n}} \\
+ & \|\left\{ (D \, [(\text{Ad }\mathcal{U})_{-x, -\xi}(A)] \circ \mathcal{F}^{-1}) \otimes I_{E_n} \right\} \, (f - v) \cdot 1_\mathscr{A} \|_{E_{2n}} \\ \leq & \frac{M \, \epsilon}{3 (M + 1) (K + 1)} + \frac{\epsilon}{3} + \frac{K \, \epsilon}{3 (M + 1) (K + 1)} < \epsilon.
\end{align*}
This completes the proof.
\end{proof}

\section{Proof of Theorem A} \label{sect:conjecture}

Fix $A \in C^\infty(\text{Ad }\mathcal{U}) \cap R_n'$ and $g \in \mathcal{S}^\mathscr{A}(\mathbb{R}^n)$. Equation \eqref{eq:unifboundedness} guarantees the existence of a constant $M > 0$ such that
\begin{equation*}
\sup_{m \in \mathbb{N}^*} \left\| \left\{ (D \, [(\text{Ad }\mathcal{U})_{-x, -\xi}(A \circ L_{\tilde{e}_m})] \circ \mathcal{F}^{-1}) \otimes I_{E_n} \right\}|_{L^2(\mathbb{R}^{2n}) \cdot 1_\mathscr{A}} \right\| \leq M.
\end{equation*}
Therefore, applying the Cauchy-Schwarz inequality to the expression \eqref{S} defining the symbol map $S$ gives, for every fixed $(x, \xi) \in \mathbb{R}^{2n}$, the estimates
\begin{align*}
& \|S(A \circ L_{\tilde{e}_m})(x, \xi)\|_\mathscr{A} \\
& \leq (2 \pi)^{n/2} \|u \cdot 1_\mathscr{A}\|_{L^2} \|\left\{ \left\{ D \, [(\text{Ad }\mathcal{U})_{-x, -\xi}(A \circ L_{\tilde{e}_m})] \circ \mathcal{F}^{-1} \right\} \otimes I_{E_n} \right\}|_{L^2(\mathbb{R}^{2n}) \cdot 1_\mathscr{A}}\| \, \|v \cdot 1_\mathscr{A}\|_{L^2} \\ 
& \leq M \, (2 \pi)^{n/2} \, \|u \cdot 1_\mathscr{A}\|_{L^2} \, \|v \cdot 1_\mathscr{A}\|_{L^2}, \qquad m \in \mathbb{N}^*.
\end{align*}
In particular,
\begin{equation*}
\|(\mathcal{R} \circ S)(A \circ L_{\tilde{e}_m})(x)\|_\mathscr{A} = \|S(A \circ L_{\tilde{e}_m})(x, 0)\|_\mathscr{A} \leq M \, (2 \pi)^{n/2} \, \|u \cdot 1_\mathscr{A}\|_{L^2} \, \|v \cdot 1_\mathscr{A}\|_{L^2},
\end{equation*}
for every $x \in \mathbb{R}^n$ and $m \in \mathbb{N}^*$. Substituting $x$ by $x - \frac{1}{2 \pi} J\xi$ in the above Equation and then multiplying both sides by the number $\|\mathcal{F}(g)(\xi)\|_\mathscr{A}$ we conclude, using the submultiplicative property of the C$^*$-norm $\|\, \cdot \,\|_\mathscr{A}$, that the estimate
\begin{equation} \label{eq:dominatedconvergence}
\|(\mathcal{R} \circ S)(A \circ L_{\tilde{e}_m}) \left( x - \frac{1}{2 \pi} J\xi \right) \, \mathcal{F}(g)(\xi)\|_\mathscr{A} \leq M \, (2 \pi)^{n/2} \, \|u \cdot 1_\mathscr{A}\|_{L^2} \, \|v \cdot 1_\mathscr{A}\|_{L^2} \, \|\mathcal{F}(g)(\xi)\|_\mathscr{A}
\end{equation}
holds, for every $(x, \xi) \in \mathbb{R}^{2n}$ and $m \in \mathbb{N}^*$. As a consequence of Proposition \ref{prop:aprox}, we have the following pointwise convergence (in $\mathscr{A}$):
\begin{equation*}
\lim_{m \rightarrow + \infty} S(A \circ L_{\tilde{e}_m})(x, \xi) = S(A)(x, \xi), \qquad (x, \xi) \in \mathbb{R}^{2n}.
\end{equation*}
Hence, using the definition of $\mathcal{R}$,
\begin{equation*}
\lim_{m \rightarrow + \infty} (\mathcal{R} \circ S)(A \circ L_{\tilde{e}_m}) \left( x - \frac{1}{2 \pi} J\xi \right) \, \mathcal{F}(g)(\xi) = (\mathcal{R} \circ S)(A) \left( x - \frac{1}{2 \pi} J\xi \right) \, \mathcal{F}(g)(\xi), \qquad (x, \xi) \in \mathbb{R}^{2n},
\end{equation*}
which when combined with the estimate in Equation \eqref{eq:dominatedconvergence} allows an application of the dominated convergence theorem \cite[Proposition 1.2.5, p.~16]{analysis-bochner}, yielding, for every fixed $x \in \mathbb{R}^n$, the equality
\begin{equation} \label{eq:dominatedconvergence2}
\lim_{m \rightarrow + \infty} \int_{\mathbb{R}^n} \left\| [(\mathcal{R} \circ S)(A \circ L_{\tilde{e}_m}) - (\mathcal{R} \circ S)(A)] \left( x - \frac{1}{2 \pi} J\xi \right) \, \mathcal{F}(g)(\xi) \right\|_\mathscr{A} \, d\xi = 0.
\end{equation}
But looking at the second equality in \eqref{L1} (which also holds for $f \in \mathcal{B}^\mathscr{A}(\mathbb{R}^n)$ -- see the Appendix \ref{sect:rieffel}) we see that \eqref{eq:dominatedconvergence2} actually implies that, for every fixed $x \in \mathbb{R}^n$,
\begin{align} \label{eq:rieffel1}
& \lim_{m \rightarrow + \infty} L_{(\mathcal{R} \circ S)(A \circ L_{\tilde{e}_m})}(g)(x) \\
& = \lim_{m \rightarrow + \infty} \frac{1}{(2 \pi)^{n / 2}} \int_{\mathbb{R}^n} (\mathcal{R} \circ S)(A \circ L_{\tilde{e}_m}) \left(x - \frac{1}{2 \pi} J\xi \right) \, \mathcal{F}(g)(\xi) \, e^{i \langle \xi, x \rangle} \, d\xi \nonumber \\
& = \frac{1}{(2 \pi)^{n / 2}} \int_{\mathbb{R}^n} (\mathcal{R} \circ S)(A) \left(x - \frac{1}{2 \pi} J\xi \right) \, \mathcal{F}(g)(\xi) \, e^{i \langle \xi, x \rangle} \, d\xi \nonumber \\
& = L_{(\mathcal{R} \circ S)(A)}(g)(x). \nonumber
\end{align}

On the other hand, it is the content of Proposition \ref{prop:mainthm1} that $A$ sends $\mathcal{S}^\mathscr{A}(\mathbb{R}^n)$ into $\mathcal{S}^\mathscr{A}(\mathbb{R}^n)$, so the set $\left\{ L_{A(\tilde{e}_m)}: m \in \mathbb{N}^* \right\}$ consists of well-defined operators on $\mathcal{S}^\mathscr{A}(\mathbb{R}^n)$ (see \eqref{rieffelop}). Moreover, since by hypothesis $A$ commutes with the operator $R_h$, for every $h \in \mathcal{S}^\mathscr{A}(\mathbb{R}^n)$, we obtain
\begin{equation*}
(A \circ L_{\tilde{e}_m})(h) = A(L_{\tilde{e}_m}(h)) = (A \circ R_h)(\tilde{e}_m) = (R_h \circ A)(\tilde{e}_m) = R_h(A(\tilde{e}_m)) = L_{A(\tilde{e}_m)}(h),
\end{equation*}
for all $m \in \mathbb{N}^*$ and $h \in \mathcal{S}^\mathscr{A}(\mathbb{R}^n)$, so 
\begin{equation*}
(\mathcal{R} \circ S)(A \circ L_{\tilde{e}_m}) = (\mathcal{R} \circ S)(L_{A(\tilde{e}_m)}) = A(\tilde{e}_m), \qquad m \in \mathbb{N}^*
\end{equation*}
(see Equation \eqref{eq:left_inverse}). Hence, $L_{(\mathcal{R} \circ S)(A \circ L_{\tilde{e}_m})} = L_{A(\tilde{e}_m)} = A \circ L_{\tilde{e}_m}$, for all $m \in \mathbb{N}^*$, so by Lemma \ref{lem:approxidentity} the equality
\begin{equation} \label{eq:rieffel2}
\lim_{m \rightarrow + \infty} L_{(\mathcal{R} \circ S)(A \circ L_{\tilde{e}_m})}(g) = \lim_{m \rightarrow + \infty} (A \circ L_{\tilde{e}_m})(g) = A(g)
\end{equation}
holds in $E_n$.

Now we must find an argument to combine Equations \eqref{eq:rieffel1} and \eqref{eq:rieffel2} and conclude that $A = L_{(\mathcal{R} \circ S)(A)}$. Since $A(g)$ belongs to $\mathcal{S}^\mathscr{A}(\mathbb{R}^n)$, Equation \eqref{eq:rieffel2} can be translated in terms of integrals:
\begin{equation*}
\lim_{m \rightarrow + \infty} \left\| \int_{\mathbb{R}^n} [L_{(\mathcal{R} \circ S)(A \circ L_{\tilde{e}_m})}(g)(x) - A(g)(x)]^*[L_{(\mathcal{R} \circ S)(A \circ L_{\tilde{e}_m})}(g)(x) - A(g)(x)] \, dx \right\|_\mathscr{A}^{1 / 2} = 0.
\end{equation*}
Fix a positive linear functional $\rho$ on $\mathscr{A}$. Then
\begin{align*}
& \lim_{m \rightarrow + \infty} \int_{\mathbb{R}^n} \left| \rho \left( [L_{(\mathcal{R} \circ S)(A \circ L_{\tilde{e}_m})}(g)(x) - A(g)(x)]^*[L_{(\mathcal{R} \circ S)(A \circ L_{\tilde{e}_m})}(g)(x) - A(g)(x)] \right) \right| dx \\
& = \lim_{m \rightarrow + \infty} \int_{\mathbb{R}^n} \rho \left( [L_{(\mathcal{R} \circ S)(A \circ L_{\tilde{e}_m})}(g)(x) - A(g)(x)]^*[L_{(\mathcal{R} \circ S)(A \circ L_{\tilde{e}_m})}(g)(x) - A(g)(x)] \right) dx = 0
\end{align*}
so, by a standard result in measure theory \cite[Theorem 3.12, p.~68]{rudin_real_and_complex}, we may extract a subsequence $(L_{(\mathcal{R} \circ S)(A \circ L_{\tilde{e}_{m_k}})}(g))_{k \in \mathbb{N}}$ (which depends on $\rho$) such that
\begin{equation*}
\left( \rho \circ \left( [L_{(\mathcal{R} \circ S)(A \circ L_{\tilde{e}_{m_k}})}(g) - A(g)]^*[L_{(\mathcal{R} \circ S)(A \circ L_{\tilde{e}_{m_k}})}(g) - A(g)] \right) \right)_{k \in \mathbb{N}} \text{ is pointwise convergent to } 0
\end{equation*}
on a subset $S_\rho \subseteq \mathbb{R}^n$ whose complement has Lebesgue measure equal to zero. Combining Equation \eqref{eq:rieffel1} with the continuity of $\rho$, we obtain
\begin{align*}
& \rho \left( [L_{(\mathcal{R} \circ S)(A)}(g)(x) - A(g)(x)]^*[L_{(\mathcal{R} \circ S)(A)}(g)(x) - A(g)(x)] \right) \\ & = \lim_{k \rightarrow + \infty} \rho \left( [L_{(\mathcal{R} \circ S)(A \circ L_{\tilde{e}_{m_k}})}(g)(x) - A(g)(x)]^*[L_{(\mathcal{R} \circ S)(A \circ L_{\tilde{e}_{m_k}})}(g)(x) - A(g)(x)] \right) = 0,
\end{align*}
for all $x \in S_\rho$. But $L_{(\mathcal{R} \circ S)(A)}(g)$ and $A(g)$ both belong to $\mathcal{S}^\mathscr{A}(\mathbb{R}^n)$, which establishes that the above equality actually holds for all $x \in \mathbb{R}^n$. Moreover, since $\rho$ is arbitrary, we get \cite[Theorem 3.3.6, p.~90]{murphy}
\begin{equation*}
[L_{(\mathcal{R} \circ S)(A)}(g)(x) - A(g)(x)]^*[L_{(\mathcal{R} \circ S)(A)}(g)(x) - A(g)(x)] = 0, \qquad x \in \mathbb{R}^n.
\end{equation*}
Hence, using the C$^*$-identity for the norm $\|\, \cdot \,\|_\mathscr{A}$ we conclude that $L_{(\mathcal{R} \circ S)(A)}(g)(x) = A(g)(x)$, for all $x \in \mathbb{R}^n$. By the arbitrariness of $g$, we conclude that $A = L_{(\mathcal{R} \circ S)(A)}$, which is exactly what we wanted to prove.

\begin{appendix}

\section{A few remarks on Rieffel's deformed algebra} \label{sect:rieffel}

As noted in the Introduction, given a skew-symmetric linear transformation $J$ on $\mathbb{R}^n$ and $f \in \mathcal{B}^\mathscr{A}(\mathbb{R}^n)$, we may define a linear operator via the iterated integral
\begin{equation} \label{eq:rieffelop_appendix}
L_f(g)(x) := \int_{\mathbb{R}^n} \left( \int_{\mathbb{R}^n} f(x + J\xi) \, g(x + y) \, e^{2 \pi i \langle \xi, y \rangle} \, dy \right) d\xi, \qquad g \in \mathcal{S}^\mathscr{A}(\mathbb{R}^n), \, x \in \mathbb{R}^n,
\end{equation}
in the sense that integration in the variable $y$ has to be performed before integration over $\xi$. In fact, since the Fourier transform maps $\mathcal{S}^\mathscr{A}(\mathbb{R}^n)$ continuously into itself \cite[p.~117]{analysis-bochner}, we see that the map $\xi \longmapsto \int_{\mathbb{R}^n} f(x + J\xi) \, g(x + y) \, e^{2 \pi i \langle \xi, y \rangle} \, dy$ belongs to $\mathcal{S}^\mathscr{A}(\mathbb{R}^n)$, so
\begin{equation*}
\int_{\mathbb{R}^n} \left\| \int_{\mathbb{R}^n} f(x + J\xi) \, g(x + y) \, e^{2 \pi i \langle \xi, y \rangle} \, dy \right\|_\mathscr{A} \, d\xi < + \infty.
\end{equation*}

Using well-known techniques in the theory of pseudodifferential operators one can write the iterated integral in Equation \eqref{eq:rieffelop_appendix} as a Bochner integral of an $\mathscr{A}$-valued function on $\mathbb{R}^{2n}$ and, as a consequence, conclude that $L_f$ maps $\mathcal{S}^\mathscr{A}(\mathbb{R}^n)$ into itself \cite[Proposition 3.3, p.~25]{rieffel}, a fact which we will now show in detail. After successive integration by parts we get, for every fixed $\xi \in \mathbb{R}^n$ and $N \geq 1$, the equality
\begin{align*}
\int_{\mathbb{R}^n} f(x + J\xi) \, g(x + y) \, e^{2 \pi i \langle \xi, y \rangle} \, dy
& = \int_{\mathbb{R}^n} f(x + J\xi) \, g(x + y) \, \frac{(1 - \Delta_y)^N e^{2 \pi i \langle \xi, y \rangle}}{(1 + 4 \pi^2 |\xi|^2)^N} \, dy \\
& = \int_{\mathbb{R}^n} f(x + J\xi) \, [(1 - \Delta_y)^N g](x + y) \, \frac{e^{2 \pi i \langle \xi, y \rangle}}{(1 + 4 \pi^2 |\xi|^2)^N} \, dy.
\end{align*}
Hence, choosing $N > n/2$ and using Equation \eqref{eq:rieffelop_appendix}, we see that $L_f(g)(x)$ equals
\begin{equation*}
\int_{\mathbb{R}^n} \int_{\mathbb{R}^n} f(x + J\xi) \, [(1 - \Delta_y)^N g](x + y) \, \frac{e^{2 \pi i \langle \xi, y \rangle}}{(1 + 4 \pi^2 |\xi|^2)^N} \, dy \, d\xi,
\end{equation*}
where the integrand is absolutely convergent in the variable $(y, \xi)$.

Repeating the procedure of integrating by parts, but with respect to the $\xi$ variable, instead, we get
\begin{align} \label{eq:rieffelop_appendix_2}
& L_f(g)(x) \\
& = \int_{\mathbb{R}^n} \int_{\mathbb{R}^n} e^{2 \pi i \langle \xi, y \rangle} \, (1 - \Delta_\xi)^M \left\{ \frac{f(x + J\xi)}{(1 + 4 \pi^2 |\xi|^2)^N} \right\} \, [(1 - \Delta_y)^N g](x + y) \, \frac{1}{(1 + 4 \pi^2 |y|^2)^M} \, dy \, d\xi \nonumber
\end{align}
for all $N > n/2$ and $M \geq 1$.

Differentiating the above formula under the integral sign we see that for each fixed $\alpha, \beta \in \mathbb{N}^n$, the expression $x^\alpha \, \partial^\beta L_f(g)(x)$ equals the sum
\begin{align*}
\sum_{|\gamma| \leq |\beta|} \binom{\beta}{\gamma} \cdot x^\alpha \cdot \int_{\mathbb{R}^n} \int_{\mathbb{R}^n} e^{2 \pi i \langle \xi, y \rangle} \, (1 - \Delta_\xi)^M
& \left\{ \frac{\partial^\gamma f(x + J\xi)}{(1 + 4 \pi^2 |\xi|^2)^N} \right\} \, \\
& \cdot [(1 - \Delta_y)^N \partial^{\beta - \gamma} g](x + y) \, \frac{1}{(1 + 4 \pi^2 |y|^2)^M} \, dy \, d\xi,
\end{align*}
for sufficiently large $M$ and $N$. Therefore,
\begin{align*}
\left\| x^\alpha \, \partial^\beta L_f(g)(x) \right\|_\mathscr{A} \leq \sum_{|\gamma| \leq |\beta|}
& \binom{\beta}{\gamma} \int_{\mathbb{R}^n} \int_{\mathbb{R}^n} \left\| (1 - \Delta_\xi)^M \left\{ \frac{\partial^\gamma f(x + J\xi)}{(1 + 4 \pi^2 |\xi|^2)^N} \right\} \right\|_\mathscr{A} \\
& \cdot \left\{ (1 + |x|^2)^{|\alpha| / 2} \, \left\| [(1 - \Delta_y)^N \partial^{\beta - \gamma} g](x + y) \right\|_\mathscr{A} \right\} \frac{1}{(1 + 4 \pi^2 |y|^2)^M} \, dy \, d\xi,
\end{align*}
Using Peetre's inequality \cite[(3.6)]{kohn-nirenberg} we get
\begin{equation*}
(1 + |x|^2)^{|\alpha| / 2} \leq 2^{|\alpha| / 2} \, (1 + |y|^2)^{|\alpha| / 2} \, (1 + |x + y|^2)^{|\alpha| / 2},
\end{equation*}
so the sum above may be majorized by the expression
\begin{align*}
2^{|\alpha| / 2} \sum_{|\gamma| \leq |\beta|} \binom{\beta}{\gamma}
& \int_{\mathbb{R}^n} \int_{\mathbb{R}^n}
\left\| (1 - \Delta_\xi)^M \left\{ \frac{\partial^\gamma f(x + J\xi)}{(1 + 4 \pi^2 |\xi|^2)^N} \right\} \right\|_\mathscr{A} \\
& \cdot \left\{ (1 + |x + y|^2)^{|\alpha| / 2} \, \left\| [(1 - \Delta_y)^N \partial^{\beta - \gamma} g](x + y) \right\|_\mathscr{A} \right\} \, \frac{(1 + 4 \pi^2 |y|^2)^{|\alpha| / 2}}{(1 + 4 \pi^2 |y|^2)^M} \, dy \, d\xi,
\end{align*}
which will be a real number as long as we choose $N > n / 2$ and $M > (n + |\alpha|) / 2$. Finally, this shows that the expression $\sup_{x \in \mathbb{R}^n} \left\| x^\alpha \, \partial^\beta L_f(g)(x) \right\|_\mathscr{A}$ may be majorized by a linear combination of terms of the form 
\begin{align*}
\sup_{x \in \mathbb{R}^n} \left\{ \|\partial^{\delta_1} f(x)\|_\mathscr{A} \right\}
& \, \sup_{x \in \mathbb{R}^n} \left\{ (1 + |x|^2)^{|\alpha| / 2} \, \|\partial^{\delta_2} g(x)\|_\mathscr{A} \right\} \\
& \cdot \int_{\mathbb{R}^n} \frac{1}{(1 + 4 \pi^2 |y|^2)^{M - |\alpha| / 2}} \, dy \, \int_{\mathbb{R}^n} \frac{1}{(1 + 4 \pi^2 |\xi|^2)^Q} \, d\xi,
\end{align*}
where $\delta_1, \delta_2 \in \mathbb{N}^n$ are multiindices and $Q \geq N$. Hence, since $f \in \mathcal{B}^\mathscr{A}(\mathbb{R}^n)$ and $g \in \mathcal{S}^\mathscr{A}(\mathbb{R}^n)$, these terms are all real numbers. This shows that $L_f$ maps $\mathcal{S}^\mathscr{A}(\mathbb{R}^n)$ continuously into itself.

Making use of oscillatory integrals (see \cite[Chapter 1]{rieffel} \cite[pp.~66--69]{cordes}), we can attribute meaning to the integral in Equation \eqref{eq:rieffelop_appendix} even when both $f$ and $g$ belong to $\mathcal{B}^\mathscr{A}(\mathbb{R}^n)$, defining Rieffel's deformed product \cite[p.~23]{rieffel} to be
\begin{equation} \label{eq:rieffelprod_appendix}
(f \times_J g)(x) := \int_{\mathbb{R}^n} \int_{\mathbb{R}^n} f(x + J\xi) \, g(x + y) \, e^{2 \pi i \langle \xi, y \rangle} \, dy \, d\xi
\end{equation}
\begin{equation*}
:= \int_{\mathbb{R}^n} \int_{\mathbb{R}^n} e^{2 \pi i \langle \xi, y \rangle} \, (1 - \Delta_\xi)^M \left\{ \frac{f(x + J\xi)}{(1 + 4 \pi^2 |\xi|^2)^N} \right\} \, [(1 - \Delta_y)^N g](x + y) \, \frac{1}{(1 + 4 \pi^2 |y|^2)^M} \, dy \, d\xi, \nonumber
\end{equation*}
where $x \in \mathbb{R}^n$ and $M, N > n/2$. It can be shown that this definition is independent of the choices of $M$ and $N$. Comparing the definition given in Equation \eqref{eq:rieffelprod_appendix} with Equation \eqref{eq:rieffelop_appendix_2} we see that, if $f \in \mathcal{B}^\mathscr{A}(\mathbb{R}^n)$ and $g \in \mathcal{S}^\mathscr{A}(\mathbb{R}^n)$, then $L_f(g) = f \times_J g$. Moreover, after differentiating under the integral sign, one sees that $f \times_J g$ also belongs to $\mathcal{B}^\mathscr{A}(\mathbb{R}^n)$.

As a consequence of a version of the Calder\'on-Vaillancourt inequality for Hilbert C$^*$-modules \cite[Theorem 3.2]{cabralforgermelo}, every pseudodifferential operator $\text{Op}(a)$ with symbol $a \in \mathcal{B}^\mathscr{A}(\mathbb{R}^{2n})$ (see Equation \eqref{eq:pseudos}) can be extended by continuity to a bounded operator (which will also be denoted by $\text{Op}(a)$) on the Hilbert $\mathscr{A}$-module $E_n$. Moreover, its operator norm satisfies
\begin{equation*}
\|\text{Op}(a)\| \leq C \, \max_{\beta, \gamma \leq \mathring{\alpha}} \sup \left\{ \|\partial_x^\beta \partial_\xi^\gamma a(x, \xi)\|_\mathscr{A}: x, \xi \in \mathbb{R}^n \right\}, \qquad \mathring{\alpha} = (1, \ldots, 1) \in \mathbb{N}^n,
\end{equation*}
for a certain constant $C > 0$. In particular, this conclusion holds for every $L_f$, $f \in \mathcal{B}^\mathscr{A}(\mathbb{R}^n)$, since $L_f$ is a pseudodifferential operator $\text{Op}(a)$ with symbol $a \in \mathcal{B}^\mathscr{A}(\mathbb{R}^{2n})$ given by $a(x, \xi) := f(x - J\xi / (2 \pi))$.

Using oscillatory integrals, we can show that the set $\left\{ \text{Op}(a): a \in \mathcal{B}^\mathscr{A}(\mathbb{R}^{2n}) \right\}$ of pseudodifferential operators on $E_n$ is a $*$-subalgebra of the C$^*$-algebra $\mathcal{L}_\mathscr{A}(E_n)$ of adjointable operators on $E_n$ (for the proof that each $\text{Op}(a)$ is an adjointable operator on $E_n$, see \cite[Proposition 3.3]{cabralforgermelo}). In fact, the restriction of the involution and composition maps to $\left\{ \text{Op}(a): a \in \mathcal{B}^\mathscr{A}(\mathbb{R}^{2n}) \right\}$ are defined, respectively, by $\text{Op}(a) \longmapsto \text{Op}(a^\dagger)$ and $\text{Op}(a) \circ \text{Op}(b) \longmapsto \text{Op}(a \times b)$, with corresponding symbols given by \cite[(3.20), (3.21)]{cabralforgermelo}
\begin{equation} \label{eq:involution_appendix}
a^\dagger(x, \xi) := \frac{1}{(2 \pi)^n} \int_{\mathbb{R}^n} \int_{\mathbb{R}^n} e^{-i \langle z, \eta \rangle} \, a(x - z, \xi - \eta)^* \, dz \, d\eta
\end{equation}
\begin{equation*} := \frac{1}{(2 \pi)^n} \int_{\mathbb{R}^n} \int_{\mathbb{R}^n} e^{-i \langle z, \eta \rangle} \, (1 + |z|^2)^{-N} (1 - \Delta_\eta)^N \Bigl\{ (1 + |\eta|^2)^{-M} (1 - \Delta_z)^M \Bigl[ a(x - z, \xi - \eta)^* \Bigr] \Bigr\} dz \, d\eta, \end{equation*}
and
\begin{equation} \label{eq:product_appendix}
(a \times b)(x, \xi) := \frac{1}{(2 \pi)^n} \int_{\mathbb{R}^n} \int_{\mathbb{R}^n} e^{-i \langle z, \eta \rangle} \, a(x, \xi - \eta) \, b(x - z, \xi) \, dz \, d\eta
\end{equation}
\begin{equation*}
:= \frac{1}{(2 \pi)^n} \int_{\mathbb{R}^n} \int_{\mathbb{R}^n} e^{-i \langle z, \eta \rangle} \, (1 + |z|^2)^{-N} (1 - \Delta_\eta)^N \Bigl\{ (1 + |\eta|^2)^{-M} (1 - \Delta_z)^M \Bigl[ a(x, \xi - \eta) \, b(x - z, \xi) \Bigr] \Bigr\} dz \, d\eta,
\end{equation*}
for all $x, \xi \in \mathbb{R}^n$, where $M, N > n/2$. Moreover, the above definitions are independent of the choices of $M$ and $N$. Specializing these formulas to the operators $L_f$, $f \in \mathcal{B}^\mathscr{A}(\mathbb{R}^n)$, shows that the product $\times_J$ is associative (see the last paragraph of \cite[Subsection ``Pseudodifferential operators with $\mathcal{C}$-valued symbols'']{cabralforgermelo}) and that Rieffel's deformed algebra $\left\{ L_f: f \in \mathcal{B}^\mathscr{A}(\mathbb{R}^n) \right\}$ is also a $*$-subalgebra of $\mathcal{L}_\mathscr{A}(E_n)$, with the involution satisfying $(L_f)^* = L_{f^*}$. Indeed, if $f$ belongs to $\mathcal{S}^\mathscr{A}(\mathbb{R}^n)$, then the function $a$ defined by $a(x, \xi) := f(x - J\xi / (2 \pi))$ belongs to $\mathcal{S}^\mathscr{A}(\mathbb{R}^{2n})$, so using Equation \eqref{eq:involution_appendix} and the fact that $\langle J \eta, \eta \rangle = 0$, for all $\eta \in \mathbb{R}^n$, we can perform the transformation $z \mapsto z + \frac{J \eta}{2 \pi}$ of the integration variable $z$ to obtain
\begin{align*}
a^\dagger(x, \xi)
& = \frac{1}{(2 \pi)^n} \int_{\mathbb{R}^n} \int_{\mathbb{R}^n} e^{-i \langle z, \eta \rangle} \, f \left( x - z - \frac{J(\xi - \eta)}{2 \pi} \right)^* \, dz \, d\eta \\
& = \frac{1}{(2 \pi)^n} \int_{\mathbb{R}^n} \int_{\mathbb{R}^n} e^{-i \langle z + \frac{J \eta}{2 \pi}, \eta \rangle} \, f \left( x - z - \frac{J \xi}{2 \pi} \right)^* \, dz \, d\eta \\
& = \frac{1}{(2 \pi)^{n/2}} \int_{\mathbb{R}^n} \mathcal{F} \left( f \left(x - \, \cdot \, - \frac{J \xi}{2 \pi} \right)^* \right)(\eta) \, d\eta = f(x - J\xi / (2 \pi))^* =: f^*(x - J\xi / (2 \pi)),
\end{align*}
for all $x, \xi \in \mathbb{R}^n$. For the general case in which $f$ belongs to $\mathcal{B}^\mathscr{A}(\mathbb{R}^n)$, let $\phi$ be a compactly supported complex-valued smooth function on $\mathbb{R}^{2n}$ satisfying $0 \leq \phi \leq 1$ which equals 1 on a neighborhood of 0 and define, for each $m \in \mathbb{N}^*$, the functions $f_m(y) := \phi \left( \frac{y}{m} \right) f(y)$, $y \in \mathbb{R}^n$, and $a_m(x, \xi) := f_m(x - J\xi / (2 \pi))$, $(x, \xi) \in \mathbb{R}^{2n}$. Then $(a_m)_{m \in \mathbb{N}^*}$ converges pointwise to $a \colon (x, \xi) \longmapsto f(x - J\xi / (2 \pi))$, so by the continuity of the involution operation on $\mathscr{A}$, we see that $((a_m)^\dagger(x, \xi) = f_m^*(x - J\xi / (2 \pi)))_{m \in \mathbb{N}^*}$ converges to $f^*(x - J\xi / (2 \pi))$, for each fixed $(x, \xi) \in \mathbb{R}^{2n}$. On the other hand, after successive integration by parts and applications of the Leibniz rule, we get
\begin{align*}
(a_m)^\dagger(x, \xi)
& = \frac{1}{(2 \pi)^n} \int_{\mathbb{R}^n} \int_{\mathbb{R}^n} e^{-i \langle z, \eta \rangle} \, f_m \left( x - z - \frac{J(\xi - \eta)}{2 \pi} \right)^* \, dz \, d\eta \\
= \frac{1}{(2 \pi)^n}
& \int_{\mathbb{R}^n} \int_{\mathbb{R}^n} e^{-i \langle z, \eta \rangle} \, (1 + |z|^2)^{-N} \\
& \cdot (1 - \Delta_\eta)^N \left\{ (1 + |\eta|^2)^{-M} (1 - \Delta_z)^M \left[ f_m \left( x - z - \frac{J(\xi - \eta)}{2 \pi} \right)^* \right] \right\} \, dz \, d\eta,
\end{align*}
which equals
\begin{align*}
\frac{1}{(2 \pi)^n}
& \int_{\mathbb{R}^n} \int_{\mathbb{R}^n} e^{-i \langle z, \eta \rangle} \, (1 + |z|^2)^{-N} \\
& \cdot (1 - \Delta_\eta)^N \left\{ (1 + |\eta|^2)^{-M} (1 - \Delta_z)^M \left[ f \left( x - z - \frac{J(\xi - \eta)}{2 \pi} \right)^* \right] \right\} \, \phi \left( \frac{x - z - \frac{J(\xi - \eta)}{2 \pi}}{m} \right) dz \, d\eta
\end{align*}
plus a linear combination of terms of the form
\begin{align*}
\frac{1}{m^{|\gamma_2|}} \cdot \frac{1}{(2 \pi)^n}
& \int_{\mathbb{R}^n} \int_{\mathbb{R}^n} e^{-i \langle z, \eta \rangle} \\
& \cdot (\partial^{\gamma_1} f) \left( x - z - \frac{J(\xi - \eta)}{2 \pi} \right)^* \, (\partial^{\gamma_2} \phi) \left( \frac{x - z - \frac{J(\xi - \eta)}{2 \pi}}{m} \right) \, g_{\gamma_1, \gamma_2}(z, \eta) \, dz \, d\eta,
\end{align*}
with $\gamma_1 \in \mathbb{N}^n$, $0 \neq \gamma_2 \in \mathbb{N}^n$ and $g_{\gamma_1, \gamma_2} \in L^1(\mathbb{R}^{2n})$.

Therefore, applying Fubini's theorem and the dominated convergence theorem, we obtain 
\begin{align*}
(a_m)^\dagger(x, \xi)
& = \frac{1}{(2 \pi)^n} \int_{\mathbb{R}^n} \int_{\mathbb{R}^n} e^{-i \langle z, \eta \rangle} \, f_m \left( x - z - \frac{J(\xi - \eta)}{2 \pi} \right)^* \, dz \, d\eta \\
& = \frac{1}{(2 \pi)^n} \int_{\mathbb{R}^n} \int_{\mathbb{R}^n} e^{-i \langle z, \eta \rangle} \, (1 + |z|^2)^{-N} \\
& \cdot (1 - \Delta_\eta)^N \left\{ (1 + |\eta|^2)^{-M} (1 - \Delta_z)^M \left[ f_m \left( x - z - \frac{J(\xi - \eta)}{2 \pi} \right)^* \right] \right\} \, dz \, d\eta \\
& \stackrel{m \rightarrow + \infty}{\longrightarrow} \frac{1}{(2 \pi)^n} \int_{\mathbb{R}^n} \int_{\mathbb{R}^n} e^{-i \langle z, \eta \rangle} \, (1 + |z|^2)^{-N} \\
& \cdot (1 - \Delta_\eta)^N \left\{ (1 + |\eta|^2)^{-M} (1 - \Delta_z)^M \left[ f \left( x - z - \frac{J(\xi - \eta)}{2 \pi} \right)^* \right] \right\} \, dz \, d\eta =: a^\dagger(x, \xi),
\end{align*}
allowing us to conclude that $a^\dagger(x, \xi) = f^*(x - J\xi / (2 \pi))$, for all $(x, \xi) \in \mathbb{R}^{2n}$.

\end{appendix}

\section*{Acknowledgements}

We would like to thank Artur Andrade, who has provided crucial ideas for obtaining the equivalence of topologies in Proposition \ref{prop:s_equals_cinfty}.

\end{document}